\documentclass[11pt,a4paper]{article}
\usepackage[left=2.6cm, top=2.6cm,bottom=2.6cm,right=2.6cm]{geometry}
\usepackage{amsmath,amssymb,amsthm,mathrsfs,calc,graphicx,stmaryrd,stmaryrd}
\usepackage[british]{babel}

\newtheorem {theorem}{Theorem}[section]
\newtheorem {proposition}[theorem]{Proposition}
\newtheorem {lemma}[theorem]{Lemma}

{\theoremstyle{definition}

}
{\theoremstyle{theorem}
\newtheorem {remark}[theorem]{Remark}

}

\newcommand{\var}{\operatorname{var}}

\def\ba{\begin{array}}
\def\ea{\end{array}}
\def\bea{\begin{eqnarray} \label}
\def\eea{\end{eqnarray}}
\def\be{\begin{equation} \label}
\def\ee{\end{equation}}
\def\bit{\begin{itemize}}
\def\eit{\end{itemize}}
\def\ben{\begin{enumerate}}
\def\een{\end{enumerate}}

\def\lan{\langle}
\def\ran{\rangle}

\def\BB{\mathbb{B}}

\def\EE{\mathbb{E}}

\def\NN{\mathbb{N}}
\def\PP{\mathbb{P}}

\def\RR{\mathbb{R}}

\def\BBd{\mathbb{B}^d}
\def\SSd{\mathbb{S}^{d-1}}

\def\a{\alpha}

\def\d{\delta}

\def\k{\kappa}
\def\l{\lambda}

\def\s{\sigma}

\def\Sig{\Sigma}

\def\bx{\mathbf{x}}

\def\cB{\mathcal{B}}
\def\cC{\mathcal{C}}

\def\cF{\mathcal{F}}

\def\cH{\mathcal{H}}
\def\cI{\mathcal{I}}

\def\cM{\mathcal{M}}
\def\cN{\mathcal{N}}

\def\cR{\mathcal{R}}

\def\cT{\mathcal{T}}
\def\cX{\mathcal{X}}

\def\sT{\mathscr{T}}

\def\dint{\textup{d}}

\def\var{{\textup{var}}}

\def\s{\otimes}

\def\cone{{\rm cone}}
\def\ext{{\rm ext}}
\def\setk{\llbracket k\rrbracket}
\def\setp{\llbracket p\rrbracket}

\begin{document}

\title{\bfseries Concentration and moderate deviations for\\ Poisson polytopes and polyhedra}

\author{Julian Grote\footnotemark[1]\ \ and Christoph Th\"ale\footnotemark[2]}

\date{}
\renewcommand{\thefootnote}{\fnsymbol{footnote}}
\footnotetext[1]{Ruhr University Bochum, Faculty of Mathematics, D-44780 Bochum, Germany. E-mail: julian.grote@rub.de}

\footnotetext[2]{Ruhr University Bochum, Faculty of Mathematics, D-44780 Bochum, Germany. E-mail: christoph.thaele@rub.de}

\maketitle

\begin{abstract}
The convex hull generated by the restriction to the unit ball of a stationary Poisson point process in the $d$-dimensional Euclidean space is considered. By establishing sharp bounds on cumulants, exponential estimates for large deviation probabilities are derived and the relative error in the central limit theorem on a logarithmic scale is investigated for a large class of key geometric characteristics. This includes the number of lower-dimensional faces and the intrinsic volumes of the random polytopes. Furthermore, moderate deviation principles for the spatial empirical measures induced by these functionals are also established using the method of cumulants. The results are applied to deduce, by duality, fine probabilistic estimates and moderate deviation principles for combinatorial parameters of a class of zero cells associated with Poisson hyperplane mosaics. As a special case this comprises the typical Poisson-Voronoi cell conditioned on having large inradius.
\bigskip
\\
{\bf Keywords}. {Convex hulls, cumulants, concentration inequalities, deviation probabilities, moderate deviation principles, Poisson hyperplanes, Poisson-Voronoi mosaics, random polytopes, zero cells.}\\
{\bf MSC}. Primary 60F10, 60D05; Secondary 52A22.
\end{abstract}

\section{Introduction}

Random polytopes are among the most classical and popular models considered in geometric probability, and their study has become a rapidly developing branch of mathematics at the borderline between geometry and probability. One reason for the increasing interest are the numerous connections and applications of random polytopes in algorithmic geometry, convex geometric analysis, optimization, random matrix theory, set estimation or multivariate statistics; we direct the reader to the surveys of B\'ar\'any \cite{BaranySurvey}, Hug \cite{HugSurvey} and Reitzner \cite{ReitznerSurvey} for further information and references.

A common method to construct a random polytope is to take the convex hull of a finite family of random points that are uniformly distributed in the interior of a prescribed convex body $K\subset\RR^d$ with $d\geq 2$. In their seminal paper, R\'enyi and Sulanke \cite{RenyiSulanke} considered the asymptotic behaviour of the mean vertex number and the mean volume (area) of such random polytopes if $d=2$, as the number of points tends to infinity. Since then, first-order asymptotic properties of geometric characteristics of random polytopes have been investigated for general space dimensions by B\'ar\'any \cite{Barany89,Barany92}, B\"or\"oczky, Hoffmann and Hug \cite{BoeroHoffmannHug}, Reitzner \cite{ReitznerStochApproxSmooth04,ReitznerCombStructure05} or Sch\"utt \cite{Schuett}, to name only a few. More recently, the focus has turned towards asymptotic second-order characteristics like the variance of the number of vertices or the variance of the volume. The classical Efron-Stein inequality has been used by Reitzner \cite{ReitznerEfromStein03} to obtain upper bounds for these variances as well as laws of large numbers. Matching lower bounds together with related central limit theorems have been shown on different levels of generality by B\'ar\'any and Reitzner \cite{BaranyReitznerVariance,BaranyReitznerCLT}, Cabo and Groeneboom \cite{CaboGroeneboom}, Groeneboom \cite{Groeneboom}, Hueter \cite{Hueter}, Pardon \cite{Pardon}, Reitzner \cite{ReitznerClt05}, Schreiber \cite{Schreiber2002} and Vu \cite{VuCLT}. This line of research has been continued in a series of remarkable papers by Calka, Schreiber and Yukich \cite{CalkaSchreiberYukich}, Calka and Yukich \cite{CalkaYukich1}, and Schreiber and Yukich \cite{SchreiberYukich}.

In contrast to the typical or `normal' behaviour of random polytopes, much less is known about their atypical or exceptional behaviour, or at scales in between. For random polytopes in the unit ball, Calka and Schreiber \cite{CalkaSchreiber} have obtained information on large deviations for the vertex number and Schreiber \cite{Schreiber2003} has computed certain moderate deviation probabilities for the mean width. Moreover, the paper of Vu \cite{VuConcentration} deals in a general context with concentration inequalities for the volume and the vertex number. Besides such large deviation or concentration inequalities, it is from a probabilistic point of view also natural to ask for the behaviour of geometric characteristics associated with random polytopes on intermediate scales `between' that of the above-mentioned law of large numbers and that of a central limit theorem. The present paper is an attempt to fill this gap and to prove a set of concentration inequalities in the case where the underlying convex body $K$ is the $d$-dimensional unit ball $\BBd$ and where the family of random points is induced by a Poisson point process (Poisson polytopes). We refer to the papers of Affentranger \cite{Affentranger}, Buchta and M\"uller \cite{BuchtaMueller}, Calka and Schreiber \cite{CalkaSchreiber}, Hsing \cite{Hsing}, K\"ufer \cite{Kuefer}, M\"uller \cite{Mueller} and Schreiber \cite{Schreiber2002,Schreiber2003} for distinguished results about random polytopes in $\BBd$.

\medskip

Consider a stationary Poisson point process in $\RR^d$ with intensity $\l>0$, let $\eta_\l$ be its restriction to $\BBd$ and let $\Pi_\lambda$ be the convex hull of the points of $\eta_\lambda$. For simplicity and to facilitate access to our results, we restrict for the rest of this introduction to the vertex number $f_0(\Pi_\lambda)$ of the random polytopes $\Pi_\lambda$ and refer to Section \ref{sec:MainResults} for theorems dealing with other geometric characteristics of $\Pi_\l$ as well. Our first theorem is a concentration inequality for the vertex number of the random polytopes $\Pi_\l$.

\begin{theorem}\label{thm:LDI}
Let $y\geq 0$. Then, for $\l\geq c_1$,
$$
\PP\big(|f_0(\Pi_\lambda)-\EE f_0(\Pi_\l)|\geq y\sqrt{\var f_0(\Pi_\l)}\,\big)\leq 2\exp\Big(-{1\over 4}\min\Big\{{y^2\over 2^{d+4}},c_2\lambda^{d-1\over 2(d+1)(d+4)}y^{1\over d+4}\Big\}\Big)
$$
with constants $c_1,c_2\in(0,\infty)$ only depending on $d$.
\end{theorem}

Theorem \ref{thm:LDI} should be compared with Theorem 2.11 in \cite{VuConcentration}. Provided that $\l$ is sufficiently large it says in our situation that
\begin{equation}\label{eq:VuF0}
\PP\big(|f_0(\Pi_\lambda)-\EE f_0(\Pi_\l)|\geq y\sqrt{\var f_0(\Pi_\l)}\,\big)\leq 2\exp\big(-b_1y^2\big)+p_{NT}\,,
\end{equation}
for all $0<y<b_2\l^{-{(d+3)^2\over(d+1)(3d+5)}}$ with constants $b_1,b_2\in(0,\infty)$ only depending on $d$, see also \cite{ReitznerSurvey} for a related version. Here, $p_{NT}$ is the probability of what is called a `non-typical event' in \cite{VuConcentration} and satisfies the estimate $
p_{NT}\leq\exp(-b_3\l^{d-1\over 3d+5})$, independently of $y$, with another constant $b_3\in(0,\infty)$ depending only on $d$. Our theorem basically recovers the exponential term in Vu's inequality. However, while Vu's inequality involves the boundary term $p_{NT}$, which does not depend on $y$, such a term is not present in Theorem \ref{thm:LDI}. Furthermore, our inequality yields an exponential estimate for all $y\geq 0$ and not only for values of $y$ close to zero. We also emphasize that the inequality in Theorem \ref{thm:LDI} remains valid for a wide class of geometric functionals, while in \cite{VuConcentration} besides of $f_0(\Pi_\l)$ only the volume of $\Pi_\l$ is treated. On the other hand, Theorem \ref{thm:LDI} deals with the case of a random polytope in the unit ball, whereas in \cite{VuConcentration} arbitrary underlying convex bodies are permitted.

\medskip

Our next result is an estimate for certain deviation probabilities on a logarithmic scale that characterize the relative error in the central limit theorem for the normalized vertex number. To state it, denote by $\Phi(\,\cdot\,)$ the distribution function of a standard Gaussian random variable.

\begin{theorem}\label{thm:MDProbabs}
For $0\leq y\leq c_5\l^{d-1\over 2(d+1)(2d+7)}$ and $\l\geq c_3$ one has that
\begin{align*}
\Bigg|\log{\PP\big(f_0(\Pi_\l)-\EE f_0(\Pi_\l)\geq y\sqrt{\var f_0(\Pi_\l)}\,\big)\over 1-\Phi({y})}\Bigg| &\leq c_4\,(1+y^3)\,\l^{-{d-1\over 2(d+1)(2d+7)}}\quad\text{and}\\
\Bigg|\log{\PP\big(f_0(\Pi_\l)-\EE f_0(\Pi_\l)\leq -y\sqrt{\var f_0(\Pi_\l)}\,\big)\over \Phi(-{y})}\Bigg| &\leq c_4\,(1+y^3)\,\l^{-{d-1\over 2(d+1)(2d+7)}}
\end{align*}
with constants $c_3,c_4,c_5\in(0,\infty)$ only depending on $d$.
\end{theorem}

Our next theorem makes a statement about moderate deviations of the rescaled vertex number of $\Pi_\l$, which can be regarded as a kind of refinement of a central limit theorem, compare with Remark \ref{rem:OtherResults}. We will see in Theorem \ref{thm:MDPScalarGeneral} below that the set $B$ appearing in Theorem \ref{thm:MDPScalar} can be replaced in a way by an arbitrary measurable subset $B\subset\RR$ and that the rescaled vertex number of the random polytope $\Pi_\l$ satisfies a so-called moderate deviation principle. 

\begin{theorem}\label{thm:MDPScalar}
Let $(a_\l)_{\l>0}$ be such that
$$
\lim_{\l\to\infty}a_\l=\infty\qquad\text{and}\qquad\lim_{\l\to\infty}a_\l\l^{-{d-1\over 2(d+1)(2d+7)}}=0\,.
$$
Then, for all $y\in\RR$, one has that
\begin{equation*}
\lim_{\l\to\infty}{1\over a_\l^{2}}\log\PP\Big({1\over a_\l}{f_0(\Pi_\l)-\EE f_0(\Pi_\l)\over\sqrt{\var f_0(\Pi_\l)}}\in B\Big) = -{y^2\over 2}\qquad\text{with}\qquad B=[y,\infty)\,.
\end{equation*}
\end{theorem}

As anticipated above, we will see in Section \ref{sec:MainResults} that Theorem \ref{thm:LDI}, Theorem \ref{thm:MDProbabs} and Theorem \ref{thm:MDPScalar} continue to hold for a large class of key geometric functionals of the random polytopes $\Pi_\l$ (possibly under different rescalings). In particular, this includes
\begin{itemize}
\item[-] the number of $j$-dimensional faces of $\Pi_\l$ for all $j\in\{0,1,\ldots,d-1\}$,
\item[-] the missed volume of $\Pi_\l$ in $\BBd$,
\item[-] the missed volume of the Voronoi-flower of $\Pi_\l$,
\item[-] the mean width of $\Pi_\l$ and, more generally, 
\item[-] the $j$-th intrinsic volume of $\Pi_\l$ for all $j\in\{1,\ldots,d-1\}$.
\end{itemize}
In addition, we will work on the level of empirical measures and thus take care also of the spatial profile of the involved functionals. This in turn puts us in a position to present our announced moderate deviation principle also on the level of measures. Let us emphasize at this point that Theorem \ref{thm:MDProbabs} and its generalization in Theorem  \ref{thm:MDProbasGeneral} as well as the moderate deviation principles in Theorem \ref{thm:MDPScalarGeneral} and Theorem \ref{thm:MDPMeasureGeneral} seem to be the first results in this direction in the context of random polytopes and that we were not able to locate counterparts in the existing literature.

\medskip

Instead of taking the convex hull of random points, it is also natural to consider random sets that arise as intersections of random half-spaces, see the surveys of Hug \cite{HugSurvey} and Reitzner \cite{ReitznerSurvey}. To understand the geometric and the combinatorial structure of such random polyhedra is of importance, for example, in linear optimization. In particular, the performance of the well-known simplex algorithm depends on the number of edges of the polyhedron that is defined as intersection of the set of half-space determined by a system of linear inequalities. One way to obtain a deeper insight into the generic combinatorial complexity that arises in such situations is to consider random polyhedral sets as argued in Borgwardt's monograph \cite{Borgwardt}. By a duality argument borrowed from the works of Calka and Schreiber \cite{CalkaSurvey,CalkaSchreiberPV} we transfer our results for random polytopes to combinatorial parameters of a certain class of random polyhedra that are associated with Poisson hyperplanes (Poisson polyhedra). In particular, this includes the prominent typical cell of a stationary Poisson-Voronoi tessellation of $\RR^d$ conditioned on having a large inradius. In this context, we also contribute to the results around D.G.\ Kendall's conjecture asking for the asymptotic geometry of `large' tessellation cells and for which we refer in particular to the paper of Hug and Schneider \cite{HugSchneider} as well as to the references cited therein.

\medskip

Let us briefly comment on the technique we use to derive Theorems \ref{thm:LDI}--\ref{thm:MDPScalar} and their generali\-zations stated in Section \ref{sec:MainResults}. It is based on precise estimates of the cumulants of the involved random variables. The methodology to deduce fine probabilistic estimates from bounds on cumulants goes back to the `Lithuanian school of probability' and is presented in the monograph of Saulis and Statulevi\v{c}ius \cite{SaulisBuch}. In the context of geometric probability this has been used by Eichelsbacher, Rai\v{c} and Schreiber \cite{EichelsbacherSchreiberRaic} to deduce results similar to those presented above for a class of so-called stabilizing functionals. However, the random polytope functionals we consider behave quite differently and are not within the reach of the results in \cite{EichelsbacherSchreiberRaic}. Instead, we use the principal idea from \cite{CalkaSchreiberYukich,SchreiberYukich} that connects $\Pi_\l$ with a parabolic growth process in the upper half-space. The key advantage of this connection lies in the fact that in the rescaled parabolic picture spatial correlations are much easier to localize and to describe. We then develop the methods from \cite{BaryshnikovYukich} and \cite{EichelsbacherSchreiberRaic} further to make the cumulant approach available in the context of the random polytopes $\Pi_\l$. Our probabilistic estimates then follow from the main `lemmas' in \cite{SaulisBuch} and the moderate deviations from the work of D\"oring and Eichelsbacher \cite{DoeringEichelsbacher}. The main technical difficulty in carrying out this approach is that only the points of the Poisson point process $\eta_\l$ close to the boundary of $\BBd$ contribute to the geometry of $\Pi_\l$, an effect that does not occur for the models considered in \cite{EichelsbacherSchreiberRaic}, but which is typical for random polytopes.

\medskip

The remaining parts of the paper are organized as follows. In Section \ref{sec:background} we introduce the formal framework and recall the necessary results from \cite{CalkaSchreiberYukich}. Our main theorems for Poisson polytopes are presented in full generality in Section \ref{sec:MainResults}, while the final Section \ref{sec:proofs} contains their proofs. In Section \ref{sec:Hyperplanes} we apply our main results to a parametric class of random polyhedra that arise from Poisson hyperplanes.

\section{Framework and background}\label{sec:background}

\subsection{Basic notions and notation}\label{sec:Basics}

\paragraph{Notation.}
In this paper we write $V_d(\,\cdot\,)$ for the $d$-dimensional volume (Lebesgue measure) of the argument set. We denote the Euclidean scalar product by $(\,\cdot\,,\,\cdot\,)$, the norm induced by it by $\|\,\cdot\,\|$, and put $\BBd:=\{x\in\RR^d:\|x\|\leq 1\}$ and $\SSd:=\{x\in\RR^d:\|x\|=1\}$. We further indicate by $\BBd(x,r)$ the ball centred at $x\in\RR^d$ with radius $r>0$ and define the constant $\k_k:=V_k(\BB^k)$, $k\in\{0,1,2,\ldots\}$. We denote by $\cH_{\SSd}^{d-1}$ the $(d-1)$-dimensional Hausdorff measure on $\SSd$. Moreover, we use the symbol $[B]$ to indicate the convex hull of a set $B\subset\RR^d$.

Let $\Sig$ be Polish space. By $\cB(\Sig)$ we denote the space of bounded measurable functions $f:\Sig\to\RR$ and we write $\cM(\Sig)$ for the space of finite signed measures on $\Sig$. For $f\in\cB(\Sig)$ and $\nu\in\cM(\Sig)$ we introduce the abbreviation
$$
\lan f,\nu\ran := \int \dint\nu\,f
$$
for the integral of $f$ with respect to $\nu$. We will further use the symbol $\cC(\Sig)$ for the space of continuous functions on $\Sig$.

\paragraph{Grassmannians.}
By $G(d,j)$ we denote the space of all $j$-dimensional linear subspaces of $\RR^d$, $j\in\{0,1,\ldots,d-1\}$. We supply $G(d,j)$ with the unique Haar probability measure $\nu_j$, see \cite{SW}. Moreover, for $L\in G(d,1)$ we put $G(L,1):=L$ and $G(L,j):=\{E\in G(d,j):L\subset E\}$ if $j\geq 2$, i.e., $G(L,j)$ is the collection of all $j$-dimensional linear subspaces of $\RR^d$ that contain the fixed line $L$. We let $\nu_j^L$ be the unique Haar probability measure on $G(L,j)$, see Chapter 7.1 in \cite{SW} for a construction. We clearly have that $\nu_1^L$ coincides with the unit-mass Dirac measure concentrated at $L$. 

\paragraph{Intrinsic volumes, $f$-vectors and extreme points.}
By a convex body we understand a compact convex subset $K\subset\RR^d$ with non-empty interior. For a convex body $K$ we denote by $V_j(K)$ the $j$th intrinsic volume of $K$, $j\in\{0,1,\ldots,d-1\}$. It is given by
$$
V_j(K) := {d\choose j}{\k_d\over\k_j\k_{d-j}}\int_{G(d,j)}\nu_j(\dint L)\, V_j(K|L)\,,
$$
where $K|L$ indicates the orthogonal projection of $K$ onto the $j$-dimensional subspace $L$ and where $V_j(K|L)$ is the $j$-volume of $K|L$, cf.\ \cite{Gruber}. In particular, $2V_{d-1}(K)$ is the surface area and $V_1(K)$ is a constant multiple of the mean width of $K$, while $V_0(K)=1$.

A polytope $P\subset\RR^d$ is the convex hull of a finite point set. For $j\in\{0,1,\ldots,d-1\}$ we write $\cF_j(P)$ for the collection of all $j$-dimensional faces of $P$ and put $f_j(P):=|\cF_j(P)|$, where $|\,\cdot\,|$ stands for the cardinality of the argument set. In particular, $\cF_0(P)$ is the set of vertices and $f_0(P)$ the vertex number of $P$. Moreover, the elements of $\cF_{d-1}(P)$ are called the facets of $P$ and $f_{d-1}(P)$ is the number of facets of $P$. The vector $(f_0(P),f_1(P),\ldots,f_{d-1}(P))$ is the so-called $f$-vector of $P$ and describes its combinatorial structure.

Recall from \cite{Gruber} that an extreme point of a convex body $K$ is a point of $K$ which does not lie in any open line segment joining two points of $K$. We write $\ext(K)$ for the set of extreme points of $K$. By the extreme points of a finite point set $\cX$ we understand the extreme points of its convex hull, i.e., $\ext(\cX):=\ext([\cX])$.

\paragraph{Poisson point processes.}
Let $\Sig$ be a Polish space. For a locally finite set $\cX$ of points in $\Sig$ and a measurable subset $B\subset\Sig$ we write $\cX(B)$ for the number of points $x\in\cX$ with $x\in B$. Let $\nu$ be a $\sigma$-finite measure on $\Sig$ without atoms. A Poisson point process $\eta$ in $\Sig$ with intensity measure $\nu$ is a locally finite set of random points in $\Sig$ with the following two properties:
\begin{itemize}
\item[-] the number $\eta(B)$ of points falling in a measurable set $B\subset\Sig$ with $\nu(B)\in(0,\infty)$ is Poisson distributed with mean $\nu(B)$,
\item[-] for $n\in\NN$ and pairwise disjoint measurable sets $B_1,\ldots,B_n\subset\Sig$ the random variables $\eta(B_1),\ldots,\eta(B_n)$ are mutually independent.
\end{itemize}
If $\Sig=\RR^d$ and $\nu$ is a constant multiple $\l\in(0,\infty)$ of the Lebesgue measure on $\RR^d$ we will speak about a stationary Poisson point process with intensity $\l$. Its distribution is invariant under the action of all (deterministic) shifts in $\RR^d$.

For a Poisson point process $\eta$ on $\Sig$ with intensity measure $\nu$ as above and a non-negative measurable function $\xi$ acting on pairs $(x,\eta)$ with $x\in\Sig$ one has that
\begin{equation}\label{eq:Mecke}
\EE\sum_{x\in\eta}\xi(x,\eta) = \int_{\Sig}\nu(\dint x)\,\EE\xi(x,\eta\cup\{x\})\,,
\end{equation}
cf.\ \cite[Theorem 3.2.5]{SW}. Identity \eqref{eq:Mecke} is the so-called Mecke equation for Poisson point processes.

\paragraph{Large and moderate deviation principles.}
One says that a family $(\nu_\l)_{\l>0}$ of probability measures on a Polish space $\Sigma$ fulfils a large deviation principle (LDP) on $\Sigma$ with speed $s_\l$ and (good) rate function $I:\Sigma\to[0,\infty]$, as $\l\to\infty$, if $\lim\limits_{\l\to\infty}s_\l=\infty$, $I$ is lower semi-continuous with compact level sets, and if
$$
-\inf_{y\in{\rm int}(B)}I(y)\leq\liminf_{\l\to\infty}{1\over s_\l}\log\nu_\l(B)\leq\limsup_{\l\to\infty}{1\over s_\l}\log\nu_\l(B)\leq-\inf_{y\in{\rm cl}(B)}I(y)
$$
for every measurable subset $B$ of $\Sigma$ with interior ${\rm int}(B)$ and closure ${\rm cl}(B)$, cf. \cite{DemboZeitouni}. 

A family $(X_\l)_{\l>0}$ of $\Sigma$-valued (and usually rescaled) random variables is said to satisfy a LDP with speed $s_\l$ and rate function $I$ if the family of their distributions does. One usually speaks about a moderate deviation principle (MDP) instead of a LDP if the rescaling of the involved random variables is between that of a law of large numbers and that of a central limit theorem.

While large deviations are strongly influenced by the distribution of the involved random variables, moderate deviations are influenced by both, the central limit theorem and the large deviation's behaviour. As for an LDP, the speed of decay of the involved probabilities in an MDP is exponential and the central limit theorem is usually reflected by the appearance of the universal Gaussian rate function $I(y)={y^2\over 2}$, which is independent of the probabilistic nature of the underlying random variables.

\subsection{The key geometric functionals}\label{sec:KeyFunctionals}

In this section we introduce the basic geometric functionals to which our main results apply. These are the missed-volume functional, the intrinsic volume functionals, the $k$-face functionals and the Voronoi-flower functional. From now on, let $\eta_\l$ be the restriction to $\BBd$ of a stationary Poisson point process with intensity $\l>0$ and denote by $\Pi_\l$ the random convex hull generated by the points of $\eta_\l$.

Let $\cX$ be a finite set of points in $\BBd$. We write $\cF_{d-1}([\cX])$ for the family of facets of the convex hull $[\cX]$ and for $f\in\cF_{d-1}([\cX])$ let $N(f)$ be the point of $f$ that is closest to the boundary $\SSd$ of $\BBd$ (if there is more than one such point, we select the first one with respect to the lexicographic ordering). Moreover, for $x\in\cX$ we define $\cF(x,\cX)$ as the collection of all facets $f\in\cF_{d-1}([\cX])$ of $[\cX]$ with $x=N(f)$ and put $\cone(\cF(x,\cX)):=\{ry:y\in \cF(x,\cX),r>0\}$. 

We start by introducing the missed-volume functional $\xi_{V_d}$. It is given by
\begin{equation}\label{eq:DefXi_r}
\xi_{V_d}(x,\cX):=V_d((\BBd\setminus [\cX])\cap\cone(\cF(x,\cX)))\,,
\end{equation}
if $x\in\ext(\cX)$ and zero otherwise.
Using $\xi_{V_d}$ we can represent the missed volume of $\Pi_\l$ in $\BBd$ as
$$
V_d(\BBd)-V_d(\Pi_\l) = \sum_{x\in\eta_\l}\xi_{V_d}(x,\eta_\l)\,.
$$

To define the intrinsic volume functionals let for $x\in\BBd\setminus\{0\}$, $L(x)\in G(d,1)$ be the line spanned by $x$ and the origin. Now, for $j\in\{1,\ldots,d-1\}$ and $y\in\BBd\setminus\{0\}$ put
\begin{equation}\label{eq:defThetak}
\theta_j(y):=\int_{G(L(y),j)}\nu_j^{L(y)}(\dint M)\,{\bf 1}(y\notin [\cX]|M)
\end{equation}
and define
\begin{equation}\label{eq:DefXi_vk}
\xi_{V_j}(x,\cX):={{d-1\choose j-1}\over\k_{d-j}}\int_{(\BB^d\setminus[\cX])\cap\cone(\cF(x,\cX))}\dint y\,\|y\|^{-(d-j)}\,\theta_j(y) 
\end{equation}
if $x\in\ext(\cX)$ and zero otherwise. 
We conclude from Lemma \ref{lem:RepVK} in the Appendix that the difference $V_j(\BBd)-V_j(\Pi_\l)$ of the $j$th intrinsic volume of $\BBd$ and $\Pi_\l$ admits the representation
$$
V_j(\BBd)-V_j(\Pi_\l) = \sum_{x\in\eta_\l}\xi_{V_j}(x,\eta_\l)\,.
$$

For $j\in\{0,1,\ldots,d-1\}$ the $j$-face functional is defined as
\begin{equation}\label{eq:DefXi_fk}
\xi_{f_j}(x,\cX):=\begin{cases}|\cF_j(x,\cX)| &: x\in\cF_0([\cX])\\ 0 &: x\notin\cF_0([\cX])\,,\end{cases}
\end{equation}
where for $x\in\cF_0([\cX])$, $\cF_j(x,\cX)$ stands for the collection of $j$-dimensional faces $f\in\cF_j([\cX])$ of $[\cX]$ with $x=N(f)$ (the definition of $N(\,\cdot\,)$ clearly extends to elements of $\cF_j([\cX])$). This implies that the number of $j$-dimensional faces of $\Pi_\l$ can be written as
$$
f_j(\Pi_\l) = \sum_{x\in\eta_\l}\xi_{f_j}(x,\eta_\l)\,.
$$

Finally, the Voronoi-flower ${\rm VF}(\cX)$ of $\cX$ is given by
$$
{\rm VF}(\cX):=\bigcup_{x\in\cX}\BBd\Big({x\over 2},{\|x\|\over 2}\Big)\,.
$$
The Voronoi-flower of a random polytope is of interest because of the following observation. Writing $h_K(u):=\max\{(u,v):v\in K\}$ for the support function of a convex body $K$ in direction $u\in\SSd$, one has that the defect support function $1-h_{[\cX]}(u)$ of $[\cX]$ is precisely the distance between $\SSd$ and ${\rm VF}(\cX)$ in direction $u$. We put
\begin{equation}\label{eq:DefXi_s}
\xi_{VF}(x,\cX):=V_d\big((\BBd\setminus {\rm VF}(\cX))\cap\cone(\cF(x,\cX))\big) 
\end{equation}
if $x\in\ext(\cX)$ and zero otherwise, and notice that
$$
V_d(\BBd\setminus{\rm VF}(\eta_\l)) = \sum_{x\in\eta_\l}\xi_{VF}(x,\eta_\l)\,,
$$
which is nothing than the integrated defect support function of $\Pi_\l$ over $\SSd$. The Voronoi-flower of $\Pi_\l$ is also a crucial object in Section \ref{sec:Hyperplanes}.

We call $\xi_{V_d},\xi_{V_j},\xi_{f_j}$ and $\xi_{VF}$ defined by \eqref{eq:DefXi_r}, \eqref{eq:DefXi_vk}, \eqref{eq:DefXi_fk} and \eqref{eq:DefXi_s}, respectively, the key geometric functionals of the random polytope $\Pi_\l$ and define $\Xi:=\{\xi_{V_d},\xi_{V_j},\xi_{f_j},\xi_{VF}\}$. It is crucial for our purposes that each of the geometric characteristics $V_d(\BBd)-V_d(\Pi_\l)$, $V_j(\BBd)-V_j(\Pi_\l)$ with $j\in\{1,\ldots,d-1\}$, $f_j(\Pi_\l)$ with $j\in\{0,\ldots,d-1\}$ and $V_d(\BBd\setminus{\rm VF}(\eta_\l))$ is representable as
$$
\sum_{x\in\eta_\l}\xi(x,\eta_\l)
$$
with some key geometric functional $\xi\in\Xi$.

\subsection{Rescaled functionals}\label{sec:RescaledFct}

Let $n:=(0,0,\ldots,1)$ be the north pole of $\SSd$ and identify the tangent space ${\rm Tan}(\SSd,n)$ of $\SSd$ at $n$ with the $(d-1)$-dimensional Euclidean space $\RR^{d-1}$. The exponential map $\exp:{\rm Tan}(\SSd,n)\to\SSd$ transforms a vector $u\in{\rm Tan}(\SSd,n)$ into a point $\exp(u)\in\SSd$ such that $\exp(u)$ lies at the end of a geodesic ray of length $\|u\|$ and direction $u$ emanating from $n$, see Figure \ref{fig2}. In particular $\exp(n)=0$. (The exponential map should not be confused with the exponential function which is denoted by the same symbol, but the meaning will always be clear from the context.) Although the exponential map is well defined on the whole tangent space, its injectivity region is $\BB^{d-1}(0,\pi):=\{u\in {\rm Tan}(\SSd,n):\|u\|<\pi\}$, the centred (open) ball in $\RR^{d-1}$ with radius $\pi$. Let us further denote by $\exp^{-1}$ the inverse of the exponential map, which is well defined on $\SSd\setminus\{-n\}$.

\begin{figure}[t]
\centering
\includegraphics[width=\columnwidth]{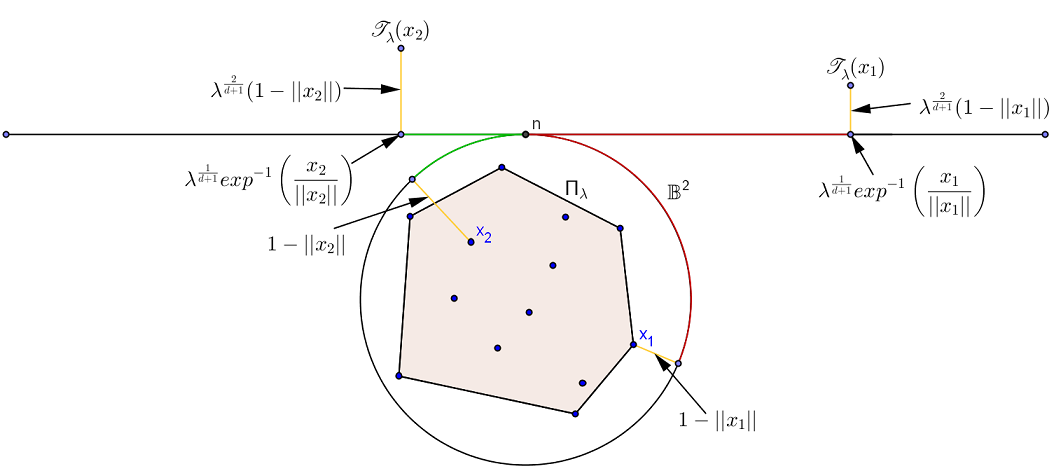}
\caption{Illustration of the scaling transformation $\sT_\l$.}
\label{fig2}
\end{figure}

Following \cite{CalkaSchreiberYukich}, we define the scaling transformation $\sT_\l$ mapping $\BBd$ to $\RR^{d-1}\times\RR_+$ by
\begin{equation}\label{eq:ScalingTrafo}
\sT_\l(x) := \Big(\l^{1\over d+1}\exp^{-1}\Big({x\over\|x\|}\Big),\l^{2\over d+1}(1-\|x\|)\Big)\,,\qquad x\in\BBd\setminus\{0\}\,.
\end{equation}
In particular, we notice that, by the well-known mapping properties of Poisson point processes, $\sT_\l$ maps the Poisson point process $\eta_\l$ to another Poisson point process $\eta_\l^\sT$ in the region
$$
\cR_\l:=\l^{1/(d+1)}\BB^{d-1}(0,\pi)\times[0,\l^{2/(d+1)})\subset\RR^{d-1}\times\RR_+
$$
(note that with probability one, neither $-n$ nor $0$ is not a point of $\eta_\l$, meaning that the definition of $\exp$ and $\exp^{-1}$ at these points is irrelevant). In what follows, we parametrize the points of $\RR^{d-1}\times\RR_+$ as pairs $(v,h)$ with $v\in\RR^{d-1}$ and $h\in\RR_+$. Using this parametrization, it is known from Equation (2.14) in \cite{CalkaSchreiberYukich} that the intensity measure of $\eta_\l^\sT$ has density
$$
(v,h)\mapsto{\sin^{d-2}(\l^{-1/(d+1)}\|v\|)\over \|\l^{-1/(d+1)}v\|^{d-2}}(1-\l^{-2/(d+1)}h)^{d-1}
$$
with respect to the Lebesgue measure on $\cR_\l$. In particular, this implies that the limit process of $\eta_\l^\sT$, as $\l\to\infty$, is a Poisson point process $\eta$ on the whole half-space $\RR^{d-1}\times\RR_+$ whose intensity measure coincides with the Lebesgue measure on that space (the convergence has to be understood in the sense of total variation distance of measures on compacts when $\eta_\l^\sT$ and $\eta$ are regarded as random counting measures).

Using the scaling transformation $\sT_\l$ we define the collection of rescaled key geometric functionals. Let $\xi\in\Xi$ and put
$$
\xi^{(\l)}(x,\cX) := \xi(\sT_\l^{-1}(x),\sT_\l^{-1}(\cX))
$$
for a locally finite point set $\cX$ in the region $\cR_\l$ and $x\in\cX$. We set $\Xi^{(\l)}:=\{\xi^{(\l)}:\xi\in\Xi\}$.

One of the crucial features of the above scaling transformation is that the rescaled functionals $\xi^{(\l)}$ exhibit a weak spatial dependence property in the following sense. A random variable $R:=R(\xi,x,\l)$ is called a radius of localization for $\xi^{(\l)}$ if, with probability one,
$$
\xi^{(\l)}(x,\eta_\l^{\sT})=\xi^{(\l)}\big(x,\eta_\l^\sT\cap {\rm Cyl}(x,r)\big)
$$
for all $r\geq R$. Here, for a point $x=(v,h)$, ${\rm Cyl}(x,r)$ stands for the cylinder $\BB^{d-1}(v,r)\times\RR_+$. It has been shown in \cite{CalkaSchreiberYukich} that $R$ has super-exponentially decaying tails uniformly in $x$ and $\l$. More formally, one can find constants $c_1,c_2\in(0,\infty)$ only depending on $\xi$ such that
\begin{equation}\label{eq:ExpStabilization}
\PP(R\geq u) \leq c_1 \exp(-c_1^{-1}u^{d+1})\,,\qquad u>0\,,
\end{equation}
for all $\l\geq c_2$, uniformly in $x$ and $\l$. In addition, \cite{CalkaSchreiberYukich} shows that the weaker estimate
\begin{equation}\label{eq:ExpStabilization2}
\PP(R\geq u) \leq c_1 \exp(-c_1^{-1}u)\,,\qquad u>0\,,
\end{equation}
is also satisfied.
Moreover, defining the individual scaling exponents $e[\xi]$ by
$$
e[\xi]:=\begin{cases} 1 &: \xi\in\{\xi_{V_d},\xi_{V_j},\xi_{VF}\}\\ 0 &: \xi=\xi_{f_j}\,,\end{cases}
$$
one has that the random variables $\l^{e[\xi]}\xi^{(\l)}(x,\eta_\l^\sT)$ have moments of all orders, i.e., for all $p\geq 1$ one has that
\begin{equation}\label{eq:MomentsAllOrders}
\sup_{\l\geq 1}\sup_{x\in\cR_\l}\sup_{|\cX|\leq p}\EE|\l^{e[\xi]}\xi^{(\l)}(x,\eta_\l^\sT\cup\cX)|^{p} <\infty\,,
\end{equation}
where the innermost supremum runs over all subsets $\cX\subset\cR_\l$ with $x\in\cX$ and at most $p$ elements.

\subsection{A generalized growth process and its scaling limit}\label{sec:GenGrowth}

With the rescaled Poisson point process $\eta_\l^\sT$ defined in the previous section one can associate what has been called a generalized growth process in \cite{CalkaSchreiberYukich,SchreiberYukich}. We denote by $d_s$ the usual geodesic distance on $\SSd$ and for $x=(v,h)\in\RR^{d-1}\times\RR_+$ we define the set
\begin{align*}
\Pi_{x,\l}^{\rm up} := \{(v',h')\in\RR^{d-1}\times\RR_+:h'\geq h+\l^{2\over d+1}(1-\cos\omega(v',v))-h'(1-\cos\omega(v',v))\}\,,
\end{align*}
where $\omega(v',v):=d_s(\exp(\l^{-{1\over d+1}}v),\exp(\l^{-{1\over d+1}}v'))$. The generalized growth process $\Psi^{(\l)}$ is now given by
$$
\Psi^{(\l)} := \bigcup_{x\in\eta_\l^\sT}\Pi_{x,\l}^{\rm up}\,.
$$
It can be used to describe the geometry of the random polytopes $\Pi_\l$. Namely, let $s_\l(u):=1-h_{\Pi_\l}(u)$ be the defect support function of $\Pi_\l$ in direction $u\in\SSd$ and let for $v\in\l^{1\over d+1}\BB^{d-1}(0,\pi)$,
$$
s_\l^\sT(v):=\l^{2\over d+1}\,s_\l\Big(\exp\Big(\l^{-{1\over d+1}}v\Big)\Big)
$$
be the rescaled defect support function. Then the lower boundary $\partial\Psi^{(\l)}$ of $\Psi^{(\l)}$ (i.e., the continuous random surface that bounds the random set $\Psi^{(\l)}$ from below) coincides with the graph of $s_\l^\sT$, see \cite{CalkaSchreiberYukich}.

We say that a particle of $\Psi^{(\l)}$ is extreme if it is not completely covered by other particles and we denote the set of extreme points of the extreme particles of $\Psi^{(\l)}$ by $\ext(\Psi^{(\l)})$. In particular, we notice that the image under the scaling transformation $\sT_\l$ of the set of extreme points of the random polytope $\Pi_\l$ coincides with $\ext(\Psi^{(\l)})$. Let us recall from \cite[Lemma 3.2]{SchreiberYukich} that the probability that a point $x=(v,h)\in\cR_\l$ belongs to $\ext(\Psi^{(\l)})$ decays exponentially with the height $h$ of $x$. More precisely, we have that there are universal constants $c_1,c_2\in(0,\infty)$ such that
\begin{equation}\label{eq:ExtremExpDecay}
\PP(x\in\ext(\Psi^{(\l)})) \leq c_1\,\exp\Big(-c_1^{-1}h^{{d+1\over 2}}\Big)
\end{equation}
for all $\l\geq c_2$, uniformly in $x$. The results in \cite{SchreiberYukich} also provide the weaker estimate
\begin{equation}\label{eq:ExtremExpDecay2}
\PP(x\in\ext(\Psi^{(\l)})) \leq c_1\,\exp\big(-c_1^{-1}h\big)\,.
\end{equation}

For completeness, we also mention that the generalized growth processes $\Psi^{(\l)}$ have a scaling limit, as $\l\to\infty$. To describe it, let $\eta$ be a Poisson point process in the upper half-space $\RR^{d-1}\times\RR_+$ whose intensity measure coincides with the Lebesgue measure. Following \cite{CalkaSchreiberYukich} we define the upward paraboloid process $\Psi$ in $\RR^{d-1}\times\RR_+$ with respect to $\eta$ as
$$
\Psi:=\bigcup_{x\in\eta}(x\oplus\Pi^{\rm up})\,,
$$
where $\oplus$ is the usual Minkowski sum and where $\Pi^{\rm up}$ stands for the upward paraboloid
$$
\Pi^{\rm up}:=\big\{(v,h)\in\RR^{d-1}\times\RR_+:h\geq \|v\|^2/2\big\}\,.
$$
In the terminology of stochastic geometry this means that $\Psi$ is a Boolean model with germ process $\eta$ and grains equal to $\Pi^{\rm up}$. We denote by $\partial\Psi$ the lower boundary of $\Psi$. Now, one has that for all $R>0$, $\partial\Psi^{(\l)}$ converges to $\partial\Psi$ on the space $\cC(\BB^{d-1}(0,R))$ supplied with the supremum norm, as $\l\to\infty$. In other words this is to say that for all $R>0$, the graph of the rescaled defect support function $s_\l^\sT$ converges on $\cC(\BB^{d-1}(0,R))$ to $\partial\Psi$, as $\l\to\infty$.

\begin{remark}\rm 
Using the upward paraboloid process $\Psi$ and its dual, the so-called paraboloid hull process introduced in \cite{CalkaSchreiberYukich}, one can explicitly describe various asymptotic expectation and variance constants of the functionals introduced in Section \ref{sec:KeyFunctionals}. However, these results are not used in what follows and for this reason we refer the reader to \cite{CalkaSchreiberYukich} for further details.
\end{remark}

\section{Main results for Poisson polytopes}\label{sec:MainResults}

Let $\eta_\l$ be the restriction to $\BBd$ of a stationary Poisson point process in $\RR^d$ with intensity $\l>0$ and let $\Pi_\l$ be the convex hull of $\eta_\l$. For a key geometric functional $\xi\in\Xi$ introduced in Section \ref{sec:KeyFunctionals} and its rescaled version $\xi^{(\l)}$ as considered in Section \ref{sec:RescaledFct} we introduce the spatial empirical measure
\begin{equation}\label{eq:DefMu}
\mu_\l^\xi := \sum_{x\in\eta_\l}\xi(x,\eta_\l)\,\d_x=\sum_{x'\in\eta_\l^\sT}\xi^{(\l)}(x',\eta_\l^\sT)\,\d_{\sT_\l^{-1}(x')}\,,
\end{equation}
where $\d_x$ stands for the unit-mass Dirac measure at $x$. The centred version of $\mu_\l^\xi$ is throughout denoted by $\bar\mu_\l^\xi:=\mu_\l^\xi-\EE\mu_\l^\xi$. We further define for $f\in\cB(\BBd)$ the quantity $\sigma_\l^\xi[f]:=(\var\lan f,\l^{e[\xi]}\mu_\l^\xi\ran)^{1/2}$ and emphasize that for each of the key geometric functionals $\xi$ we consider, there exists a constant $c>0$ only depending on $d$ and $\xi$ such that for all $\l\geq c$ one has that
\begin{equation}\label{eq:VarLowerBound}
\sigma_\l^\xi[f]\geq C\,\lan f^2,\cH_{\SSd}^{d-1}\ran^{1\over 2}\,\l^{{d-1\over 2(d+1)}}
\end{equation}
if $f\in\cC(\BBd)$ with another constant $C\in(0,\infty)$ that depends only on $d$ and on $\xi$. This follows from the variance considerations in \cite{BaranyReitznerVariance,CalkaSchreiberYukich,ReitznerClt05} and we point out that the continuity of $f$ has essentially been used there to derive the lower variance bound \eqref{eq:VarLowerBound}. For this reason we also assume continuity of $f$ in our results.

\medskip

Our first result is a general concentration inequality for integrals with respect to the empirical measure induced by our key geometric functionals; proofs are postponed to Section \ref{sec:proofs}. In particular, for $f\in\cC(\BBd)$ we shall derive an exponential estimate for the probability
$$
\PP\big(|\lan f,\l^{e[\xi]}\bar\mu_\l^\xi\ran|\geq y\,\sigma_\l^\xi[f]\big)\,,
$$
where, recall, $e[\xi]$ is the scaling exponent of the functional $\xi$. Theorem \ref{thm:LDI} presented in the introduction is a special case of this result. We also define the individual weights
\begin{equation}\label{eq:DEF}
w[\xi] := \begin{cases} 2 &: \xi\in\{\xi_{V_1},\ldots,\xi_{V_d},\xi_{VF}\} \\ j &: \xi=\xi_{f_j}\ \text{for some}\ j\in\{0,\ldots,d-1\}\end{cases}
\end{equation}
of the key geometric functionals that originate from the moment condition in Lemma \ref{lem:MomentsAllOrders} below and that appear in all our findings.

\begin{theorem}\label{thm:LDIGeneral}
Let $f\in\cC(\BBd)$, $\xi\in\Xi$ and $y\geq 0$. Suppose that $\lan f^2,\cH_{\SSd}^{d-1}\ran\neq 0$. Then, for $\l\geq c_1$,
\begin{equation}\label{eq:LDIGeneral}
\PP\big(|\lan f,\l^{e[\xi]}\bar\mu_\l^\xi\ran|\geq y\,\sigma_\l^\xi[f]\big)\leq 2\exp\Bigg(-{1\over 4}\min\Big\{{y^2\over 2^{d+w[\xi]+4}},c_2\l^{d-1\over 2(d+1)(d+w[\xi]+4)}y^{1\over d+w[\xi]+4}\Big\}\Bigg)
\end{equation}
with constants $c_1,c_2\in(0,\infty)$ only depending on $d$ and $\xi$, or on $d$, $\xi$ and $f$, respectively.
\end{theorem}

Theorem \ref{thm:LDIGeneral}  should be related to the existing results in the literature. In case that $\xi=\xi_{f_0}$, $\lan 1,\l^{e[\xi]}\bar\mu_\l^\xi\ran$ is the centred vertex number and a detailed discussion has already been presented in the introduction. The only other result in the literature we are aware of is a concentration inequality in \cite{VuConcentration} for the missed volume (and its closely related version in \cite{ReitznerSurvey}). Its structure is basically the same as that of the corresponding inequality \eqref{eq:VuF0} for the vertex number. In particular, this estimate contains a boundary term $p_{NT}$ which does not depend on $y$ and is valid only for arguments $y$ in a certain range around zero that depends on $\l$. 

\medskip

In contrast to Theorem \ref{thm:LDIGeneral}, for the next results we could not locate counterparts in the existing literature. To the best of our knowledge, Theorem \ref{thm:MDProbasGeneral} as well as the moderate deviation principles in Theorem \ref{thm:MDPScalarGeneral} and Theorem \ref{thm:MDPMeasureGeneral} seem to be the first results in this direction in the context of random polytopes. We start with the asymptotic behaviour of deviation probabilities related to the relative error in the central limit theorem. More precisely, for a key geometric functional $\xi\in\Xi$ and $f\in\cC(\BBd)$ we are seeking for bounds on the relative error
\begin{equation}\label{eq:RelativeError}
{\PP(\lan f,\l^{e[\xi]}\bar\mu_\l^{\xi}\ran\geq \sigma_\l^\xi[f]\,y)\over 1-\Phi(y)}\,,
\end{equation}
where $\Phi(\,\cdot\,)$ is the distribution function of the standard Gaussian random variable. In addition, we are interested in conditions on $y=y(\l)$ in terms of $\l$ under which the expression in \eqref{eq:RelativeError} converges to $1$, as $\l\to\infty$. It is readily seen that Theorem \ref{thm:MDProbabs} presented in the introduction is a special case of the next theorem. 

\begin{theorem}\label{thm:MDProbasGeneral}
Let $f\in\cC(\BBd)$ with $\lan f^2,\cH_{\SSd}^{d-1}\ran\neq 0$ and $\xi\in\Xi$. For $0\leq y\leq c_5\l^{d-1\over 2(d+1)(2(d+w[\xi])+7)}$ and $\l\geq c_6$ one has that
\begin{align*}
\Bigg|\log{\PP(\lan f,\l^{e[\xi]}\bar\mu_\l^{\xi}\ran\geq \sigma_\l^\xi[f]\,y)\over 1-\Phi(y)}\Bigg| &\leq c_7\,(1+y^3)\,\l^{-{d-1\over 2(d+1)(2(d+w[\xi])+7)}}\quad\text{and}\\
\Bigg|\log{\PP(\lan f,\l^{e[\xi]}\bar\mu_\l^{\xi}\ran\leq -\sigma_\l^\xi[f]\,y)\over \Phi(-y)}\Bigg| &\leq c_7\,(1+y^3)\,\l^{-{d-1\over 2(d+1)(2(d+w[\xi])+7)}}
\end{align*}
with constants $c_5,c_6,c_7\in(0,\infty)$ only depending on $d$, on $\xi$ and on $f$.
\end{theorem}

\begin{remark}\rm 
Our methods also allow to derive precise estimates for the relative error in \eqref{eq:RelativeError}, which involve the so-called Cram\'er-Petrov series, cf.\ \cite{SaulisBuch}. To keep the result simple and to avoid heavy notation, we decided to state it here in a form which suppresses higher-order terms of the asymptotic exponential expansion.
\end{remark}

After having investigated large and moderate deviation probabilities, we turn now to a moderate deviation principle in a partial intermediate regime of rescalings between that of a central limit theorem and a law of large numbers. In particular, for all key geometric functionals $\xi\in\Xi$, functions $f\in\cC(\BBd)$ and for sets $B\subset\RR$ of the form $B=[y,\infty)$ with $y\in\RR$ we will see that
$$
\lim_{\l\to\infty}{1\over a_\l^{2}}\log\PP\Big({1\over a_\l}{\lan f,\l^{e[\xi]}\bar\mu_\l^{\xi}\ran\over\sigma_\l^\xi[f]}\in B\Big) = -{y^2\over 2}
$$
for all rescalings $a_\l$ satisfying the growth condition \eqref{eq:GrowthMDP} below. It is clear that Theorem \ref{thm:MDPScalar} in the introduction is a special case of this result.

\begin{theorem}\label{thm:MDPScalarGeneral}
Let $f\in\cC(\BBd)$ with $\lan f^2,\cH_{\SSd}^{d-1}\ran\neq 0$ and $\xi\in\Xi$. Further, let $(a_\l)_{\l>0}$ be such that
\begin{equation}\label{eq:GrowthMDP}
\lim_{\l\to\infty}a_\l=\infty\qquad\text{and}\qquad\lim_{\l\to\infty}a_\l\l^{-{d-1\over 2(d+1)(2(d+w[\xi])+7)}}=0\,.
\end{equation}
Then $\Big({1\over a_\l}{\lan f,\l^{e[\xi]}\bar\mu_\l^{\xi}\ran\over\sigma_\l^\xi[f]}\Big)_{\l>0}$ satisfies a moderate deviation principle on $\RR$ with speed $a_\l^2$ and rate function $I(y)={y^2\over 2}$.
\end{theorem}

Our final aim is to lift the result of Theorem \ref{thm:MDPScalarGeneral} to a moderate deviation principle on the level of measures (a so-called level-2 MDP). To do so, we first need to introduce the necessary topological notions. The weak topology on $\cM(\SSd)$ is generated by the sets $U_{f,x,\delta}:=\{\nu\in\cM(\SSd):|\lan f,\nu\ran-x|<\epsilon\}$ with $x\in\RR$, $\epsilon>0$ and $f\in\cC(\SSd)$, see \cite[Chapter 6.2]{DemboZeitouni}. It is also known from \cite{DemboZeitouni} that $\cM(\SSd)$ supplied with the weak topology is a locally convex, Hausdorff topological vector space whose topological dual is identified with the collection of linear functionals $\nu\mapsto\lan f,\nu\ran$, $f\in\cC(\SSd)$.

To present our result, we recall from Theorem 7.1 in \cite{CalkaSchreiberYukich} that for all $\xi\in\Xi$ there exists a constant $\sigma_\infty^\xi\in(0,\infty)$ such that
$$
\lim_{\l\to\infty}\l^{-{d-1\over d+1}}\,\var\lan f,\l^{e[\xi]}\mu_\l^\xi\ran=(\sigma_\infty^\xi)^2\,\lan f^2,\cH_{\SSd}^{d-1}\ran
$$
if $f\in\cC(\BBd)$ (the strict positivity of $\sigma_\infty^\xi$ follows from \cite[Corollary 7.1]{CalkaSchreiberYukich} and from \cite[Lemma 8]{ReitznerClt05}).

\begin{theorem}\label{thm:MDPMeasureGeneral}
Let $\xi\in\Xi$ and let $(a_\l)_{\l>0}$ be such that \eqref{eq:GrowthMDP} is satisfied. Then the family $\Big({1\over a_\l}{\l^{e[\xi]}\bar\mu_\l^\xi\over \sigma_\infty^\xi\,\l^{{(d-1)/(2(d+1))}}}\Big)_{\l>0}$ satisfies a moderate deviation principle on $\cM(\SSd)$, supplied with the weak-topology, with speed $a_\l^2$ and rate function
$$
I(\nu) = \begin{cases} {1\over 2}\lan\varrho^2,\cH_{\SSd}^{d-1}\ran &: \nu\ll \cH_{\SSd}^{d-1}\text{ with density }\varrho={\dint\nu\over\dint\cH_{\SSd}^{d-1}}\\ \infty &: \text{otherwise}\,.
\end{cases}
$$
\end{theorem}

In Theorem \ref{thm:MDPScalarGeneral} and Theorem \ref{thm:MDPMeasureGeneral} we have seen \textit{partial} MDPs, covering a part of the regime of scalings between the central limit theorem and the law of large numbers; the \textit{full} range would correspond to all scalings $a_\l$ with $a_\l\to\infty$ and $a_\l\l^{-(d-1)/(2(d+1))}\to 0$, as $\l\to\infty$. However, following the discussion in \cite{EichelsbacherSchreiberRaic}, we may argue that there are examples of weakly dependent spatial random systems known in the literature that satisfy a MDP with a Gaussian rate function only up to some critical regime of rescalings beyond that of the central limit theorem. For this reason, it might well be the case that for at least some of the key geometric functionals of the random polytopes we consider there is no full-range Gaussian MDP. We also refer to Remark \ref{rem:Optimality} below.

\medskip

It is a natural question whether our results presented above continue to hold for underlying convex bodies other than $\BBd$. The paper \cite{CalkaSchreiberYukich} establishes variance asymptotics and central limit theorems for the aforementioned key geometric functionals of $\Pi_\l$. In a later paper \cite{CalkaYukich1} the authors show that for some of these functionals the variance asymptotics and central limit theorems can be transferred to the situation in which the unit ball is replaced by a convex body with sufficiently smooth boundary. The proof is involved and highly technical. We expect that also some of our results could -- with presumably much effort -- be transferred using the methods established in \cite{CalkaYukich1}. However, to keep the length of the paper within bounds, we have decided to restrict to the representative case of the unit ball, which is also needed in the next section.

\begin{remark}\label{rem:OtherResults}\rm 
Besides of the results presented above, our method can be used to deduce further information about the random polytopes $\Pi_\l$.
\begin{itemize}
\item[(i)] (Central limit theorems) The cumulant bound presented in Proposition \ref{prop:CumulantEstimate} below together with Corollary 2.1 in \cite{SaulisBuch} yields the following Berry-Esseen estimate. For $\xi\in\Xi$, $f\in\cC(\BBd)$ with $\lan f^2,\cH_{\SSd}^{d-1}\ran\neq 0$ and a standard Gaussian random variable $Z\sim\cN(0,1)$ one has that
$$
\sup_{y\in\RR}\Bigg|\PP\Bigg({\lan f,\l^{e[\xi]}\bar\mu_\l^\xi\ran\over\sigma_\l^\xi[f]}\leq y\Bigg)-\PP(Z\leq y)\Bigg|\leq c\,\l^{-{d-1\over 2(d+1)(2(d+w[\xi])+7)}}
$$
with a constant $c\in(0,\infty)$ only depending on $\xi$, on $d$ and on $f$. In particular, as $\l\to\infty$, the random variables $(\sigma_\l^\xi[f])^{-1}\lan f,\l^{e[\xi]}\bar\mu_\l^\xi\ran$ satisfy a central limit theorem. However, the rate of convergence we get is weaker than that obtained in \cite{CalkaSchreiberYukich,ReitznerClt05} using Stein's method.

\item[(ii)] (Higher moments and cumulants) Proposition \ref{prop:CumulantEstimate} directly yields that for $\xi\in\Xi$ and $f\in\cB(\BBd)$ the $k$th-order cumulant of $\lan f,\l^{e[\xi]}\bar\mu_\l^\xi\ran$ is upper bounded by a constant multiple of $\l^{{d-1\over d+1}}$. In contrast, one has that the $k$th moment of $\lan f,\l^{e[\xi]}\mu_\l^\xi\ran$ satisfies $\EE(\lan f,\l^{e[\xi]}\mu_\l^\xi\ran)^k\leq c\l^{{k\over 2}{d-1\over d+1}}$ with a constant $c\in(0,\infty)$ only depending on $\xi$, $f$ and $d$. The proof is the same as that in \cite{VuConcentration}, where a similar behaviour of the moments has been observed for the missed volume and the vertex number of $\Pi_\l$.

\item[(iii)] (Multivariate extensions) Consider a random vector of the form
$$
\bigg({\lan f_1,\l^{e[\xi]}\bar\mu_\l^\xi\ran\over\sigma_\l^\xi[f_1]},\ldots,{\lan f_n,\l^{e[\xi]}\bar\mu_\l^\xi\ran\over\sigma_\l^\xi[f_n]}\bigg)\,,\qquad \xi\in\Xi\,,n\in\NN\,,f_1,\ldots,f_n\in\cC(\BBd)\,.
$$
It is possible to derive a multivariate MDP for the sequence of these random vectors similarly as in \cite{BaryshnikovEichelsbacherEtAl}, but we will not develop this point here.
\end{itemize}
\end{remark}

\begin{remark}\label{rem:Optimality}\rm 
We do not claim that our results are optimal. To improve them using our methods, one would have to optimize the exponent $d+w[\xi]+4$ at $k!$ appearing in Proposition \ref{prop:CumulantEstimate} below. However, for us it is not clear, which (optimal) exponent should be expected, even not in special cases. It is also not clear whether the exponent can be chosen independenly of the space dimension $d$.
\end{remark}

\section{Applications to Poisson polyhedra}\label{sec:Hyperplanes}

We are now going to apply the results obtained in the previous section to a class of Poisson polyhedra that arise as cells of a Poisson hyperplane mosaic. To define them, fix a parameter $\a\geq 1$ and let $\nu_\a$ be the measure on $\RR^d$ that is given by the relation
$$
\int_{\RR^d}\nu_\a(\dint x)\,f(x) = \int_{\RR^d\setminus\BB^d}\dint x\,\|x\|^{\a-d}\,f(x)\,,
$$
where $f\in\cB(\RR^d)$ is non-negative. Now, let $\zeta_\l$ be a Poisson point process on $\RR^d$ with intensity measure $\lambda\nu_\a$ and notice that $\zeta_\l(\BBd)=0$ with probability one. We associate with $\zeta_\l$ a family $\zeta_\l^H$ of random hyperplanes in $\RR^d$ as follows. For $x\in\zeta_\l$ let $H_x$ be the hyperplane with unit normal vector $x/\|x\|$ and distance $\|x\|/2$ to the origin. By the mapping properties of Poisson point processes, $\zeta_\l^H$ is a Poisson point process on the space of hyperplanes in $\RR^d$. The random hyperplanes of $\zeta_\l^H$ dissect the space into random polyhedra and the principal object of our investigations is the almost surely bounded random polyhedron $Z_\l=Z_\l(\alpha)$ which contains the origin, i.e.,
$$
Z_\l := \bigcap_{x\in\zeta_\l}H_x^+\,,
$$
where $H_x^+$ denotes the half-space bounded by $H_x$ that contains the origin. This parametric family of random polyhedra has attracted considerable interest in recent years because of its connections to high-dimensional convex geometry and to a version of the famous problem of D.G.\ Kendall asking for the asymptotic geometry of `large' mosaic cells, see \cite{CalkaSurvey,CalkaSchreiberPV,HH,HHRT,HugSchneider,SW}. It includes the following special case that has received particular attention and is well known in the literature, cf.\ \cite{SW}. It is concerned with a stationary Poisson-Voronoi mosaic. To define it, let $\eta$ be a stationary Poisson point process in $\RR^d$ with unit intensity. For each $x\in\eta$ we define the Voronoi cell
$$
v(x,\eta) := \{z\in\RR^d:\|z-x\|\leq \|z-x'\|\text{ for all }x'\in\eta\}
$$
as the set of all points in $\RR^d$ that are closer to $x$ than to any other point of $\eta$. The collection of all Voronoi cells forms the Poisson-Voronoi mosaic. Its typical cell can intuitively be understood as randomly chosen (and then shifted to the origin) from the set of all Voronoi cells, where each cell has the same chance of being selected, independently of size and shape. As a consequence of Slivnyak's theorem for Poisson point processes, it can be defined as
$$
Z^{\rm PV} := \{z\in\RR^d:\|z\|\leq\|z-x\|\text{ for all }x\in\eta\}\,,
$$
i.e., as the Voronoi cell of the origin, see \cite{SW}. By the inradius $R_{\rm in}(Z^{\rm PV})$ of $Z^{\rm PV}$ we understand the radius of the largest ball centred at the origin that is contained in $Z^{\rm PV}$ and we denote by $Z_r^{\rm PV}$ the typical Poisson-Voronoi cell conditioned on the event that $R_{\rm in}(Z^{\rm PV})\geq r$ for some $r\geq 1$, rescaled by a factor $r^{-1}$. It is remarkable that its distribution coincides with that of the random polyhedron $Z_\l$ under the particular choice $\a=d$ and $\lambda=(2r)^d$, cf.\ \cite{CalkaSurvey,CalkaSchreiberPV}. It is known from these works that
\begin{equation}\label{eq:MkVk}
\EE f_j(Z_r^{\rm PV}) \sim m_j\,r^{d(d-1)\over d+1}\qquad\text{and}\qquad \var f_j(Z_r^{\rm PV}) \sim v_j^2\,r^{d(d-1)\over d+1}\,,\qquad j\in\{0,\ldots,d-1\}\,,
\end{equation}
with constants $m_j,v_j\in(0,\infty)$ depending only on $d$ and on $j$, where we write $f(r)\sim g(r)$ for two functions $f,g:\RR\to\RR$ if $f(r)/g(r)\to 1$, as $r\to\infty$. These relations describe the first- and second-order asymptotic combinatorial complexity of typical Poisson-Voronoi cells with large inradius. Furthermore, asymptotic normality of $f_j(Z_r^{\rm PV})$ has also been obtained in \cite{CalkaSurvey,CalkaSchreiberPV}. (The results in these papers are formulated only for the case $j=d-1$ and in \cite{CalkaSchreiberPV} even for $d=2$, but the extension to arbitrary $j\in\{0,\ldots,d-1\}$ and $d$ is straight forward.)

We are also interested in the combinatorial structure of the random polyhedra $Z_\l$ and use the duality between $Z_\l$ and the random polytopes $\Pi_\l$ developed in \cite{CalkaSurvey,CalkaSchreiberPV} to derive a concentration inequality, explicit bounds for the relative error in the central limit theorem as well as a moderate deviation principle for $f_j(Z_\l)$. This adds to the various known contributions around Kendall's problem, see \cite{CalkaSurvey,CalkaSchreiberPV,HugSchneider} and the references cited therein. Since the results we obtain are formally the same as in Section \ref{sec:MainResults} with $\lan f,\l^{e[\xi]}\bar\mu_\l^\xi\ran$ there replaced by $f_j(Z_\l)$, we state (and prove) them here for particularly attractive Poisson-Voronoi case $\alpha=d$ and $\l=(2r)^d$ only.

\begin{figure}[t]
\begin{center}
\includegraphics[width=0.75\columnwidth]{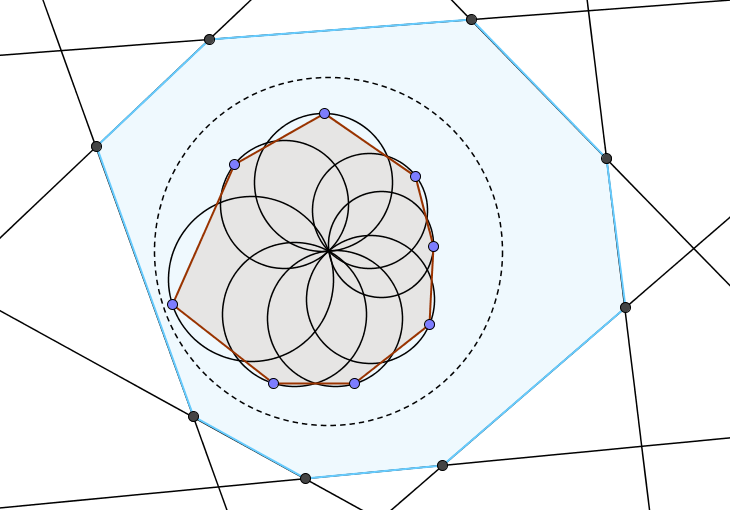}
\end{center}
\caption{The rescaled zero cell $Z_r^{\rm PV}$ (blue), the random polytope $\Pi_r$ (red) together with its associated Voronoi-flower. The unit sphere $\SSd$ is indicated by the dashed curve.}
\label{fig:hyp}
\end{figure}

\begin{theorem}\label{thm:HyperplanesGeneral}
Let $j\in\{0,\ldots,d-1\}$.
\begin{itemize}
\item[(i)] There are constants $c_1,c_2\in(0,\infty)$ only depending on $d$ and $j$, such that, for $r\geq \max\{c_1,1\}$,
$$
\PP\big(\big|r^{-{d(d-1)\over d+1}}f_j(Z_r^{\rm PV})-m_j\big|\geq y\,v_j\big)\leq 2\,\exp\Big(-{1\over 4}\min\Big\{{y^2\over 2^{d+j+4}},c_2r^{d(d-1)\over 2(d+1)(d+j+4)}y^{1\over d+j+4}\Big\}\Big)
$$
for all $y\geq 0$.
\item[(ii)] For $0\leq y\leq c_3 r^{d(d-1)\over 2(d+1)(2(d+j)+7)}$ and $r\geq\max\{c_4,1\}$ one has that
\begin{align*}
\Bigg|\log{\PP\big(r^{-{d(d-1)\over d+1}}f_j(Z_r^{\rm PV})-m_j\geq y\,v_j\big)\over 1-\Phi({y})}\Bigg| &\leq c_5\,(1+y^3)\,r^{-{d(d-1)\over 2(d+1)(2(d+j)+7)}}\quad\text{and}\\
\Bigg|\log{\PP\big(r^{-{d(d-1)\over d+1}}f_j(Z_r^{\rm PV})-m_j\leq -y\,v_j\,\big)\over \Phi(-{y})}\Bigg| &\leq c_5\,(1+y^3)\,r^{-{d(d-1)\over 2(d+1)(2(d+j)+7)}}
\end{align*}
with constants $c_3,c_4,c_5\in(0,\infty)$ only depending on $d$ and on $j$.
\item[(iii)] Suppose that $(a_r)_{r\geq 1}$ satisfies
$$
\lim_{r\to\infty}a_r=\infty\qquad\text{and}\qquad\lim_{r\to\infty}a_r\,r^{-{d(d-1)\over 2(d+1)(2(d+j)+7)}}\,.
$$
Then
$$
\bigg({1\over a_r}{f_j(Z_r^{\rm PV})-m_jr^{d(d-1)\over d+1}\over v_j\,r^{d(d-1)/(2(d+1))}}\bigg)_{r\geq 1}
$$
satisfies a moderate deviation principle on $\RR$ with speed $a_r^2$ and rate function $I(y)={y^2\over 2}$.
\end{itemize}
\end{theorem}
\begin{proof}
Consider the inversion
$$
\cI(x) := {x\over\|x\|^2}\,,\qquad x\in\RR^d\setminus\{0\}\,,
$$
and observe that the image of $\RR^d\setminus Z_r^{\rm PV}$ under $\cI$ coincides with the Voronoi-flower ${\rm VF}(\Pi_r)$ of the random polytope $\Pi_r$ in $\BB^d$ that is generated by a Poisson point process in $\BBd$ with intensity measure $(2r)^d\tilde{\nu}$. Here, the measure $\tilde{\nu}$ on $\BBd$ is given by
$$
\int_{\BBd}\tilde{\nu}(\dint x)\,f(x) = \int_{\BB^d}\dint x\,\|x\|^{-2d}\,f(x)\,,\qquad f\in\cB(\BBd)\,,f\geq 0\,.
$$
We notice now that our results presented in Section \ref{sec:MainResults} remain valid if the stationary Poisson point process there is replaced by a Poisson point processes $\eta_r$ with intensity measure $(2r)^d\tilde{\nu}$. The reason for this is that for points $x\in\BBd$ that are close to the boundary $\SSd$ of $\BBd$, i.e., for which $\|x\|$ is close to $1$, $(2r)^d\tilde{\nu}(\dint x)$ is close to $(2r)^d\dint x$, and that for sufficiently large $r$ the boundary of $\Pi_r$ is concentrated in a small annulus around $\SSd$ with overwhelming probability (we omit the formal check and refer to \cite{CalkaSurvey,CalkaSchreiberPV}).

The inversion $\cI$ induces for each $j\in\{0,\ldots,d-1\}$ a one-to-one correspondence between the sets $\cF_j(Z_r^{\rm PV})$ and $\cF_{d-j-1}(\Pi_r)$. Namely, an element of $\cF_j(Z_r^{\rm PV})$ arises almost surely as intersection of $d-j$ hyperplanes from $\zeta_{(2r)^d}^H$. They are mapped under $\cI$ to $d-j$ balls
$$
\BB^d\Big({x_1\over 2},{\|x_1\|\over 2}\Big)\,,\ldots,\BB^d\Big({x_{d-j}\over 2},{\|x_{d-j}\|\over 2}\Big)\qquad\text{with}\qquad x_1,\ldots,x_j\in\ext(\eta_r)\,.
$$
Denoting by $\partial{\rm VF}(\Pi_r)$ the boundary of the Voronoi-flower associated with $\Pi_r$ we have that
$$
\BB^d\Big({x_1\over 2},{\|x_1\|\over 2}\Big)\cap\ldots\cap\BB^d\Big({x_{d-j}\over 2},{\|x_{d-j}\|\over 2}\Big)\cap \partial{\rm VF}(\Pi_r)\neq\emptyset
$$
if and only if $[x_1,\ldots,x_{d-j-1}]\in\cF_j(Z_r^{\rm PV})$, see Figure \ref{fig:hyp}. This implies that $f_j(Z_r^{\rm PV})=f_{d-j-1}(\Pi_r)$ almost surely and we are in the position to apply Theorem \ref{thm:LDIGeneral}, Theorem \ref{thm:MDProbasGeneral} and Theorem \ref{thm:MDPScalarGeneral} with $\l=(2r)^d$, $f\equiv 1$ and $\xi=\xi_{f_{d-j-1}}$ there in combination with \eqref{eq:MkVk} to deduce the result.
\end{proof}

\begin{remark}\rm 
Using the connection between the missed volume $V_d(Z_r^{\rm PV}\setminus\BB^d(0,r))$ of the conditional typical Poisson-Voronoi cell and the first intrinsic volume difference $V_1(\BBd)-V_1(\Pi_r)$ from \cite{Schreiber2003} we can also derive a concentration inequality, error bounds in the central limit theorem and a moderate deviation principle for $V_d(Z_r^{\rm PV}\setminus\BB^d(0,r))$. This adds to the expectation and variance asymptotics, to the central limit theorem as well as to the moderate-deviation-type results proved in \cite{CalkaSchreiberPV} for the special case $d=2$.
\end{remark}

\begin{remark}\rm 
The random polyhedra $Z_\l(\a)$ contain another interesting special case. Namely, if $\a=1$ and $\lambda=r$, then $rZ_\l$ has the same distribution as the zero cell (i.e., the almost surely uniquely determined cell that contained the origin in its interior) of a stationary and isotropic Poisson hyperplane mosaic conditioned on having inradius $\geq r$, see \cite{SW}. Also in this situation, a result similar to Theorem \ref{thm:HyperplanesGeneral} is available; we leave the details to the reader.
\end{remark}

\section{Proofs of the main results for Poisson polytopes}\label{sec:proofs}

Within this section all constants $\l_0,c,c_1,c_2,\ldots$ are strictly positive, finite and such that they only depend on the space dimension $d$ and the key geometric functional $\xi$ we consider, unless otherwise specified.

\subsection{Preparations}

Fix a key geometric functional $\xi\in\Xi$ associated with the random polytopes $\Pi_\l$ and let $f\in\cB(\BBd)$. We define the sequence $(M_\l^k)_{k\geq 1}$ of moment measures of the rescaled empirical measure $\l^{e[\xi]}\mu_\l^\xi$ as at \eqref{eq:DefMu} by the relation
$$
\EE\exp(\lan f,\l^{e[\xi]}\mu_\l^\xi\ran) = 1+\sum_{k=1}^\infty{1\over k!}\lan f^{\s k},M_\l^k\ran\,,
$$
where the $k$-fold tensor product $f^{\s k}\in\cB((\BBd)^k)$ is given by $f^{\s k}(x_1,\ldots,x_k)=f(x_1)\cdots f(x_k)$. Note that $M_\l^k$ is a measure on the product space $(\BBd)^k$. Although this is not visible in our notation, we emphasize that $M_\l^k$ depends on $\xi$, but we suppress this dependency and consider $\xi$ as fixed. Now, we describe the density of these moment measures. In order to do this, we introduce for $\l>0$ the singular differential $\bar\dint_\l$ by the relation
$$
\int_{(\BBd)^k}\bar\dint_\l(x_1,\ldots,x_k) \,g(x_1,\ldots,x_k)=\l \int_{\BBd}\dint x\,g(x,\ldots,x)\,,\qquad g\in\cB((\BBd)^k)\,.
$$
Furthermore, for ${\bf x}=(x_1,\ldots,x_k)\in(\BBd)^k$ we formally put
$$
\tilde\dint_\l\bx := \sum_{L_1,\ldots,L_p\preceq\setk}\bar\dint_\l \bx_{L_1}\ldots\bar\dint_\l \bx_{L_p}\,,
$$
where $L_1,\ldots,L_p\preceq\setk$ indicates that the sum runs over all unordered partitions $\{L_1,\ldots,L_p\}$ of the set $\setk:=\{1,\ldots,k\}$ and $\bx_{L_i}:=(x_\ell:\ell\in L_i)$ for $i\in\{1,\ldots,p\}$. We can then conclude by a repeated application of the Mecke equation \eqref{eq:Mecke} that the moment measure $M_\l^k(\dint\bx)$ is absolutely continuous with respect to $\tilde\dint_\l\bx$ and has density $m_\l(\bx)$ given by
\begin{align*}
m_\l(\bx)=m_\l(x_1,\ldots,x_k) &= \EE\Big[\prod_{i=1}^k \l^{e[\xi]}\xi(x_i,\eta_\l\cup\{x_1,\ldots,x_k\})\Big]\\
&=\EE\Big[\prod_{i=1}^k \l^{e[\xi]}\xi^{(\l)}(\sT_\l(x_i),\eta_\l^\sT\cup\{\sT_\l(x_1),\ldots,\sT_\l(x_k)\})\Big]
\end{align*}
with $\bx=(x_1,\ldots,x_k)\in(\BBd)^k$, see \cite[Proposition 3.1]{EichelsbacherSchreiberRaic}.
Note that finiteness of $m_\l$ is ensured by H\"older's inequality together with \eqref{eq:MomentsAllOrders}.

Closely related to the moment measures $M_\l^k$ are the so-called cumulant measures associated with $\l^{e[\xi]}\mu_\l^\xi$. The sequence $(c_\l^k)_{k\geq 1}$ of these cumulant measures is defined via the well-known relation between moments and cumulants as
\begin{equation}\label{eq:DefCumulantMeasure}
c_\l^k := \sum_{L_1,\ldots,L_p\preceq\setk}(-1)^{p-1}(p-1)!\,M_\l^{|L_1|}\otimes\cdots\otimes M_\l^{|L_p|}\,,
\end{equation}
where $\otimes$ denotes the operation that forms the product measure. Note that $c_\l^k$ is a signed measure on the product space $(\BBd)^k$, which depends on the choice of $\xi$. It is the sequence of cumulant measures rather than that of the moment measures of the empirical measure $\l^{e[\xi]}\mu_\l^\xi$ which plays a key role in our further investigations.

Following \cite{BaryshnikovYukich,EichelsbacherSchreiberRaic}, we finally define for non-empty disjoint sets $S,T\subseteq\NN$ the (semi-) cluster measure $U_\l^{S,T}$ on $(\BBd)^{|S\cup T|}$ by
$$
U_\l^{S,T} := M_\l^{|S\cup T|}-M_\l^{|S|}\otimes M_\l^{|T|}\,.
$$
These cluster measures appear in the following decomposition of $c_\l^k$. Namely, for a non-trivial partition $\{S,T\}\preceq\setk$ one has that
\begin{equation}\label{eq:DecomCumulant}
c_\l^k = \sum_{S',T',K_1,\ldots,K_m\preceq\setk}c(S',T',K_1,\ldots,K_m)\,U_\l^{S',T'}\otimes M_\l^{|K_1|}\otimes\cdots\otimes M_\l^{|K_m|}\,,
\end{equation}
where $\{S',T',K_1,\ldots,K_m\}$ is a partition of $\setk$ with $S'\subseteq S$ and $T'\subseteq T$. The numerical coefficients $c(S',T',K_1,\ldots,K_m)$ in \eqref{eq:DecomCumulant} are known to satisfy the estimate
\begin{equation}\label{eq:EstimateCoeffDecom}
\sum_{S',T',K_1,\ldots,K_m\preceq\setk}|c(S',T',K_1,\ldots,K_m)| \leq 2^k\,k!
\end{equation}
and this upper bound is best possible according to Corollary 3.1 and Lemma 3.2 in \cite{EichelsbacherSchreiberRaic}. The representation \eqref{eq:DecomCumulant} together with the estimate \eqref{eq:EstimateCoeffDecom} are the starting point of the proof of our main results. We emphasize that although the starting point of our proof is the same as for the results in \cite{BaryshnikovYukich} or \cite{EichelsbacherSchreiberRaic}, the further details differ significantly because of the different nature of the functionals we consider.
 
\subsection{Cumulant estimates}\label{sec:CumEstimate}

This section contains the most technical part of the proof of our main theorems. The key result is the following bound for the integrals of a test function with respect to the cumulant measures introduced in the previous section. Note that the continuity of the test functions is not needed in this part of the proof. It will enter later when Proposition \ref{prop:CumulantEstimate} is combined with the variance lower bound \eqref{eq:VarLowerBound}. Recall the definition \eqref{eq:DEF} of the individual weights $w[\xi]$ of the key geometric functionals $\xi\in\Xi$.

\begin{proposition}\label{prop:CumulantEstimate}
Let $\xi\in\Xi$, $f\in\cB(\BBd)$ and $k\in\{3,4,\ldots\}$. Then, for $\lambda\geq c_1$,
$$
|\lan f^{\s k},c_\l^k\ran|\leq c_2\,c_3^k\,\|f\|_\infty^k\,\l^{{d-1\over d+1}}\,(k!)^{d+w[\xi]+4}
$$
with constants $c_1,c_2,c_3\in(0,\infty)$ only depending on $\xi$ and on $d$.
\end{proposition}

We divide the proof of Proposition \ref{prop:CumulantEstimate} into a couple of lemmas. To simplify the notation, for the remainder of this section we fix $k\in\{3,4,\ldots\}$. The next lemma will be used several times in what follows.

\begin{lemma}\label{lem1:Integral}
Let $a,b\in(0,\infty)$ and $d,p\in\NN$ with $d\geq 2$. Then
$$
\int_a^\infty \dint t\,t^{(d-1)(p-1)}\exp(-bt) = \sum_{i=1}^{(d-1)(p-1)+1}{((d-1)(p-1))_{(i-1)}\over b^i}\,\exp(-ab)\,a^{(d-1)(p-1)-(i-1)}\,,
$$
where for $n\in\NN$ and $i\in\{-1,\ldots,n\}$, $(n)_{(i)}:={n!\over(n-i)!}$ and $(n)_{(-1)}:=1$ stands for the $i$th falling factorial of $n$.
\end{lemma}
\begin{proof}
This can be shown by straight-forward repeated integration-by-parts. We omit the details.
\end{proof}

We need the following lemma that refines the moment condition \eqref{eq:MomentsAllOrders}.

\begin{lemma}\label{lem:MomentsAllOrders}
\begin{itemize}
\item[(i)] For $\l\geq c_{4}$ and $x\in\BBd$ one has that 
$$
\EE\big|\l^{e[\xi]}\xi^{(\l)}(\sT_\l(x),\eta_\l^\sT\cup\{\sT_\l(x)\})\big|^p \leq c_5\,c_{6}^p\,(p!)^{w[\xi]}
$$
for all integers $p\geq 1$.
\item[(ii)] For $\l\geq c_{7}$, integers $p\geq 1$ and $x_1'=(v_1,h_1),\ldots,x_p'=(v_p,h_p)\in\cR_\l$ one has that
$$
E_1:=\EE\Big(\prod_{i=1}^p\l^{e[\xi]}\xi^{(\l)}\Big(x_i',\eta_\l^\sT\cup\bigcup_{j=1}^p\{x_j'\}\Big)\Big)^2\leq c_8\,c_{9}^p\,(p!)^{2w[\xi]}
$$
and
$$
E_2:=\EE\Big(\prod_{i=1}^p \lambda^{e[\xi]} \xi^{(\lambda)}\Big(x_i',\Big(\eta_\l^\sT\cup\bigcup_{j=1}^p\{x_j'\}\Big)\cap {\rm Cyl}\Big(x_i',{\d\over 2}\Big)\Big)\Big)^{2}\leq c_{10}\,c_{11}^p\,(p!)^{2w[\xi]}\,,
$$
where $\d:=\min_{i,j=1,\ldots,p}\|v_i-v_j\|$.
\end{itemize}
\end{lemma}
\begin{proof} We start with part (i) and consider the missed-volume functional $\xi=\xi_{V_d}$. We can assume that $\sT_\l(x)$ is an extreme point of $\Psi^{(\l)}$, since otherwise $\xi^{(\l)}(x,\,\cdot\,)$ is zero. Now, we notice that in this case and for sufficiently large $\l$ the random variable $\l^{e[\xi]}\xi^{(\l)}(\sT_\l(x),\eta_\l^\sT\cup\{\sT_\l(x)\})$ is bounded by $R(x)S(x)$, the volume of a cylinder with height $S(x)$ whose base is a $(d-1)$-dimensional ball with radius $R(x)$. Here, $S(x)=\sup_{w\in\BB^{d-1}(v,R(x))}\partial\Psi^{(\l)}(w)$, $R(x)$ stands for the radius of localization of $\xi^{(\l)}$ at $\sT_\l(x)$ and $v$ is the spatial coordinate of $x$ under the scaling transformation $\sT_\l$. Using \eqref{eq:ExpStabilization2} and (a simplified version of) \cite[Equation (4.5)]{CalkaSchreiberYukich}, we conclude that, for sufficiently large $\l$,
\begin{align*}
&\EE\big|\l^{e[\xi]}\xi^{(\l)}(\sT_\l(x),\eta_\l^\sT\cup\{\sT_\l(x)\})\big|^p = p\int_0^\infty\dint s\, s^{p-1}\,\PP(\l^{e[\xi]}\xi^{(\l)}(\sT_\l(x),\eta_\l^\sT\cup\{\sT_\l(x)\})\geq s)\\
&\qquad\leq p\,c_{12}\int_0^\infty\dint r\int_0^\infty\dint s\, s^{p-1}\exp(-c_{13}\,(s/r))\,\exp(-c_{14}\,r)\\
&\qquad\leq c_{15}\,c_{16}^p\,(p!)^2\,.
\end{align*}
From the inequalities $\xi_{V_1},\ldots,\xi_{V_{d-1}}\leq\xi_{V_d}$ and $\xi_{VF}\leq\xi_{V_d}$ we obtain the result of part (i) for $\xi_{V_1},\ldots,\xi_{V_{d-1}}$ and $\xi_{VF}$.

Finally, let $\xi=\xi_{f_j}$ for some $j\in\{0,\ldots,d-1\}$. Let $x\in\BBd$ and $R$ be the radius of localization of $\xi_{f_j}$ at $\sT_\l(x)$. By $N$ we denote the number of extreme points of $\Psi^{(\l)}$ in ${\rm Cyl}(\sT_\l(x),R)$. According to \cite{CalkaSchreiberYukich}, the random variable $N$ has exponentially decaying tails uniformly in $\l$, whenever $\l$ is sufficiently large. Now, if $j=0$, then $\xi_{f_0}\leq 1$ and hence
$$
\EE\big|\xi_{f_0}^{(\l)}(\sT_\l(x),\eta_\l^\sT\cup\{\sT_\l(x)\})\big|^p\leq 1
$$
for all $p\geq 1$. If $j\in\{1,\ldots,d-1\}$, we observe that the number of $j$-dimensional faces meeting at $x$ is bounded by ${1\over j+1}{N\choose j}\leq N^j$. So,
\begin{align*}
&\EE\big|\xi_{f_j}^{(\l)}(\sT_\l(x),\eta_\l^\sT\cup\{\sT_\l(x)\})\big|^p\leq \EE N^{jp} \leq \int_0^\infty jp\,s^{jp-1}\,\PP(N\geq s)\,\dint s\\
&\leq \int_0^\infty jp\,s^{jp-1}\,c_{17}\,e^{-c_{18}s}\,\dint s\leq c_{19}\,c_{20}^p\,(pj)! \leq c_{19}\,c_{21}^p\,((p)!)^j\,,
\end{align*}
where we used the inequality $(pj)!\leq (j^{j})^p\,(p!)^j$. The proof of (i) is thus complete.

For part (ii) we first have from the proof of Lemma 3.2 in \cite{SchreiberYukich} that $E_2\leq E_1$ and that $E_1$ is bounded from above by
$$
\prod_{i=1}^p\big(\EE\big(\l^{e[\xi]}\xi^{(\l)}(x_i,\eta_\l^{\sT}\cup\{x_i\})\big)^{2p}\big)^{1/p}\,.
$$
For sufficiently large $\l$ this can now be estimated by means of part (i) and the result then follows from the fact that $(2p)!\leq 4^p\,(p!)^2$.
\end{proof}

\begin{remark}\rm
The proof above shows that one can take $c_5=\ldots=c_{11}=1$ in Lemma \ref{lem:MomentsAllOrders} if $\xi=\xi_{f_0}$ is the vertex counting functional.
\end{remark}

Our next result formalizes the intuition that the cluster measures $U_\l^{S,T}$ capture the spatial correlations of the rescaled key-geometric functionals. In particular, we show that these correlations decay exponentially fast.

\begin{lemma}\label{lem2:ExpClustering}
Let $\{S,T\}$ be a non-trivial partition of $\setk$ and $\xi\in\Xi$ be a key geometric functional. Then for $x_1'=(v_1,h_1),\ldots,x_k'=(v_k,h_k)\in\cR_\l$ and $\l\geq c_{22}$ one has that
\begin{equation}\label{eq:ExpClutering1}
\big|m_\l(\bx_{S\cup T}')-m_\l(\bx_S')m_\l(\bx_T')\big|\leq c_{23}\,c_{24}^k\,k\,(k!)^{w[\xi]}\,\exp\big(-c_{25}\max\big\{\d,h_1,\ldots,h_k\big\}\big)\,.
\end{equation}
Here,
$$
\d:=\min_{s\in S,t\in T}\|v_s-v_t\|
$$
is the separation of the points in $\RR^{d-1}$, and
$$
m_\l(\bx_S') := \EE\Big[\prod_{s\in S}\l^{e[\xi]}\xi^{(\l)}\Big(x_s',\eta_\l^\sT\cup\bigcup_{s'\in S}\{x_{s'}'\}\Big)\Big]
$$
with $m_\l(\bx_T')$ and $m_\l(\bx_{S\cup T}')$ defined similarly.
\end{lemma}
\begin{proof}
Define the random variables 
\begin{align*}
X&:= \prod_{s\in S} \lambda^{e[\xi]} \xi^{(\lambda)}\Big(x_s',\eta_\l^\sT\cup\bigcup_{s'\in S}\{x_{s'}'\}\Big)\,,\qquad 
Y:= \prod_{t\in T} \lambda^{e[\xi]} \xi^{(\lambda)}\Big(x_t', \eta_\l^\sT\cup\bigcup_{t'\in T}\{x_{t'}'\}\Big)\,,\\
W&:= \prod_{r\in S\cup T} \lambda^{e[\xi]} \xi^{(\lambda)}\Big(x_r', \eta_\l^\sT\cup\bigcup_{r'\in S\cup T}\{x_{r'}'\}\Big)
\end{align*}
and
\begin{align*}
X_\d&:= \prod_{s\in S} \lambda^{e[\xi]} \xi^{(\lambda)}\Big(x_s',\Big(\eta_\l^\sT\cup\bigcup_{s'\in S}\{x_{s'}'\}\Big)\cap {\rm Cyl}\Big(x_s',{\d\over 2}\Big)\Big)\,,\\ 
Y_\d &:= \prod_{t\in T} \lambda^{e[\xi]} \xi^{(\lambda)}\Big(x_t', \Big(\eta_\l^\sT\cup\bigcup_{t'\in T}\{x_{t'}'\}\Big)\cap {\rm Cyl}\Big(x_t',{\d\over 2}\Big)\Big)\,,\\
W_\d &:= \prod_{r\in S\cup T} \lambda^{e[\xi]} \xi^{(\lambda)}\Big(x_r', \Big(\eta_\l^\sT\cup\bigcup_{r'\in S\cup T}\{x_{r'}'\}\Big)\cap {\rm Cyl}\Big(x_r',{\d\over 2}\Big)\Big)\,,
\end{align*}
where, recall, for a point $x=(v,h)$, ${\rm Cyl}(x,r)$ stands for the cylinder $\BB^{d-1}(v,r)\times\RR_+$. Since for $s\in S$ and $t\in T$, ${\rm Cyl}\big(x_s', \frac{\delta}{2}\big)\cap {\rm Cyl}\big(x_t', \frac{\delta}{2}\big)=\emptyset$ by definition of the separation $\d$ we have, by independence, that $\EE W_\delta = \EE X_\delta\,\EE Y_\delta$ and hence
\begin{align}\label{eq:ZwischenR}
\nonumber & m_\lambda(\bx_{S\cup T}') - m_\lambda(\bx_{S}') m_\lambda(\bx_{T}') = \EE W - \EE X\,\EE Y\\
\nonumber &\qquad= \EE W_\delta-\EE X_\delta\,\EE Y_\delta+\EE(W-W_\delta) - \EE X_\delta\, \EE(Y-Y_\delta) - \EE Y\,\EE(X-X_\delta)\\
&\qquad= \EE(W-W_\delta) - \EE X_\delta\,\EE(Y-Y_\delta) - \EE Y\,\EE(X-X_\delta)\,.
\end{align}

Let $N_S$ denote the event that the radius of localization of at least one $x_s'$ with $s\in S$ exceeds $\d/2$. On the complement of $N_S$ we clearly have that $X$ coincides with $X_\delta$. We thus obtain from H\"older's inequality that
\begin{align*}
\EE|X-X_\delta| = \EE|X_\delta {\bf 1}_{N_S}|\leq (\EE X_\delta^2)^{1/2}\,\PP(N_S)^{1/2}\,.
\end{align*}
The moment in the first factor is bounded by $c_{26}^{|S|}((|S|)!)^{2w[\xi]}$ by Lemma \ref{lem:MomentsAllOrders} and the probability is bounded by $c_{27}\,|S|\,\exp(-c_{28}\,\d)$ in view of the exponential localization property \eqref{eq:ExpStabilization2}. Thus,
\begin{equation}\label{eq:xxx1}
\EE|X-X_\delta| \leq c_{29}\,c_{30}^k\,\sqrt{|S|}\,((|S|)!)^{w[\xi]}\,\exp\big(-c_{31}\,\d\big)\leq c_{29}\,c_{30}^k\,k\,((|S|)!)^{w[\xi]}\,\exp\big(-c_{31}\,\d\big)
\end{equation}
with a similar estimate also for $\EE|Y-Y_\delta|$ and $\EE|W-W_\delta|$, since $|S|,|T|,|S\cup T|\leq k$. 

Next, we denote by $N_E$ the event that for all $r\in S\cup T$, $x_r$ is an extreme point of the generalized growth process $\Psi^{(\l)}$. If only one of the points $x_r$ does not satisfy this property, the difference between the corresponding expectations is zero, since $\xi^{(\l)}(x_r,\,\cdot\,)$ is zero. Replacing the event $N_S$ above by $N_E$ yields in view of the exponential decay property \eqref{eq:ExtremExpDecay2} that
\begin{equation}\label{eq:xxx2}
\EE|X-X_\delta| \leq c_{32}\,c_{33}^k\,k\,((|S|)!)^{w[\xi]}\,\exp\big(-c_{34}\,\max\big\{h_1,\ldots,h_k\big\}\big)\,,
\end{equation}
again with similar estimates also for $\EE|Y-Y_\delta|$ and $\EE|W-W_\delta|$. Combining \eqref{eq:ZwischenR} with \eqref{eq:xxx1} and \eqref{eq:xxx2} finally allows us to conclude from Lemma \ref{lem:MomentsAllOrders} that
\begin{align*}
& \left|m_\lambda(\bx_{S\cup T}') - m_\lambda(\bx_{S}') m_\lambda(\bx_{T}')\right|\\
& \leq c_{35}\,c_{36}^k\,k\,\exp\big(-c_{37}\,\max\big\{\d,h_1,\ldots,h_k\big\}\big)\,\big(((|S\cup T|)!)^{w[\xi]}+2((|S|)!(|T|)!)^{w[\xi]}\big)\\
& \leq c_{35}\,c_{36}^k\,k\,\exp\big(-c_{37}\,\max\big\{\d,h_1,\ldots,h_k\big\}\big)\,\big((k!)^{w[\xi]}+2(k!)^{w[\xi]}\big)\\
& \leq c_{38}\,c_{36}^k\,k\,(k!)^{w[\xi]}\,\exp\big(-c_{37}\,\max\big\{\d,h_1,\ldots,h_k\big\}\big)\,.
\end{align*}
This completes the proof.
\end{proof}

\begin{remark}\rm 
Lemma \ref{lem2:ExpClustering} is a modification of Lemma 5.2 in \cite{BaryshnikovYukich} or Lemma 3.3 in \cite{EichelsbacherSchreiberRaic}, which exhibits a characteristic feature of random polytopes that is not present in the aforementioned papers. In particular, Lemma \ref{lem2:ExpClustering} shows that, in the rescaled picture, only points close to the (tangent) hyperplane $\RR^{d-1}$ contribute to $\mu_\l^\xi$, while points with a large height coordinate can asymptotically be neglected.
\end{remark}


We define the diagonal $\Delta:=\{(x,\ldots,x)\in(\BBd)^k:x\in\BBd\}$ and for $\bx\in(\BBd)^k$ with $\bx=(x_1,\ldots,x_k)$ and $\sT_\l(x_1)=(v_1,h_1),\ldots,\sT_\l(x_k)=(v_k,h_k)\in\cR_\l$ we put
$$
\d(\bx):=\d(x_1,\ldots,x_k):=\min\{d(v_S,v_T):\{S,T\}\preceq\setk\}\,,
$$
where $d(v_S,v_T):=\min\{\|v_s-v_t\|:s\in S,t\in T\}$. We now observe that $(\BBd)^k\setminus\Delta$ can be written as a disjoint union of sets $\sigma(\{S,T\})\subseteq(\BBd)^k$ with non-trivial partitions $\{S,T\}\preceq\setk$ in such a way that $\d(\bx)=d(v_S,v_T)$ for all $\bx\in\sigma(\{S,T\})$, see \cite{BaryshnikovYukich}. As a consequence,  $\lan f^{\s k},c_\l^k\ran$ can be decomposed as follows:
\begin{equation}\label{eq:CumulantSumRep}
\lan f^{\s k},c_\l^k\ran = \int_{\Delta} \dint c_\l^k\, f^{\s k}+\sum_{S,T\preceq\setk}\int_{\sigma(\{S,T\})}\dint c_\l^k\, f^{\s k}\,.
\end{equation}
We consider both terms in \eqref{eq:CumulantSumRep} separately and start with the diagonal term. To state the result, let us define
$$
v[\xi] := \begin{cases} 2 &: \xi\in\{\xi_{V_1},\ldots,\xi_{V_d},\xi_{VF}\}\\ 2j &: \xi=\xi_{f_j}\text{ for some }j\in\{0,\ldots,d-1\}\,. \end{cases}
$$

\begin{lemma}\label{lem3:Diagonal}
For $\xi\in\Xi$ and $f\in\cB(\BBd)$ one has that
$$
\Bigg|\int_{\Delta} \dint c_\l^k\, f^{\s k}\Bigg|\leq c_{38}\,c_{39}^k\,(k!)^{v[\xi]}\,\|f\|_\infty^k\,\l^{{d-1\over d+1}}
$$
for all $\l\geq c_{40}$.
\end{lemma}
\begin{proof}
By definition \eqref{eq:DefCumulantMeasure} of the cumulant measures we have that
\begin{align*}
\int_\Delta \dint c_\lambda^k\,f^{\s k} &= \sum_{L_1,\ldots,L_p\preceq\setk} (-1)^{(p-1)} (p-1)! \,\int_\Delta \dint(M_\lambda^{|L_1|}\otimes\cdots\otimes M_\lambda^{|L_p|})\,f^{\s k}\\
&= \sum_{L_1,\ldots,L_p\preceq\setk} (-1)^{(p-1)} (p-1)!\, \int_\Delta\tilde\dint_\l\bx \, m_\lambda(\bx_{L_1})\cdots m_\lambda(\bx_{L_p})\,f^{\s k}(\bx)\,.
\end{align*}    
Since we are integrating over the diagonal $\Delta$, $\bx$ is of the form $(x,\ldots,x)$ for some $x\in\BBd$ and we can only have $p=1$ in the above sum. We thus have that
$$
\Bigg|\int_{\Delta}\dint c_\l^k\,f^{\s k}\Bigg|\leq \|f\|_\infty^k\,\l\int_{\BB^d}\dint x\,\Bigg|\EE\Big(\l^{e[\xi]}\xi^{(\l)}\Big(\sT_\l(x),\eta_\l^\sT\cup\{\sT_\l(x)\}\Big)\Big)^k\Bigg|\,.
$$
We notice that $\xi^{(\lambda)}$ is different from zero if and only if $\sT_\l(x)$ is an extreme point of $\Psi^{(\lambda)}$. Thus, using the Cauchy-Schwarz inequality, Lemma \ref{lem:MomentsAllOrders} and the exponential decay property \eqref{eq:ExtremExpDecay2} we find that, for $\xi\in\{\xi_{V_1},\ldots,\xi_{V_d},\xi_{VF}\}$,
\begin{align*}
&\Bigg|\EE\Big(\l^{e[\xi]}\xi^{(\l)}\Big(\sT_\l(x),\eta_\l^\sT\cup\{\sT_\l(x)\}\Big)\Big)^k\Bigg|\\
&\qquad = \Bigg|\EE\Big(\l^{e[\xi]}\xi^{(\l)}\Big(\sT_\l(x),\eta_\l^\sT\cup\{\sT_\l(x)\}\Big)\,{\bf 1}(\sT_\l(x)\in\ext(\Psi^{(\l)}))\Big)^k\Bigg|\\
& \qquad\leq c_{41}\,c_{42}^k\,(2k)!\, \exp\big(-c_{43} h\big)\\
& \qquad\leq c_{41}\,c_{44}^k\,(k!)^{2}\, \exp\big(-c_{43} h\big)
\end{align*} 
for sufficiently large $\l$, since $(2k)!\leq 4^k\,(k!)^2$. Here, $h$ is the height coordinate of $x$ under the transformation $\sT_\l$.
Similarly, using Lemma \ref{lem:MomentsAllOrders} with $\xi=\xi_{f_j}$ for some $j\in\{0,\ldots,d-1\}$ we find that
\begin{align*}
&\Bigg|\EE\Big(\l^{e[\xi]}\xi^{(\l)}\Big(\sT_\l(x),\eta_\l^\sT\cup\{\sT_\l(x)\}\Big)\Big)^k\Bigg|\leq c_{41}\,c_{44}^k\,(k!)^{2j}\, \exp\big(-c_{43} h\big)
\end{align*} 
and thus
\begin{align*}
&\Bigg|\EE\Big(\l^{e[\xi]}\xi^{(\l)}\Big(\sT_\l(x),\eta_\l^\sT\cup\{\sT_\l(x)\}\Big)\Big)^k\Bigg|\leq c_{41}\,c_{44}^k\,(k!)^{v[\xi]}\, \exp\big(-c_{43} h\big)
\end{align*} 
for all $\xi\in\Xi$.  Integrating this expression over $\BBd$ by introducing spherical coordinates and taking into account the definition \eqref{eq:ScalingTrafo} of $\sT_\l$ yields that
\begin{align*}
&\Bigg|\int_{\Delta}\dint c_\l^k\,f^{\s k}\Bigg|\leq c_{45}\,c_{46}^k\,(k!)^{v[\xi]}\,\|f\|_\infty^k\,\l\int_{\SSd}\cH_{\SSd}^{d-1}(\dint u)\\
&\qquad\qquad\qquad\qquad\qquad\times\int_0^{\l^{2\over d+1}}\dint h\,\exp\big(-c_{47} h\big)\,\l^{-{2\over d+1}}(1-\l^{-{2\over d+1}}h)\\
&\qquad\leq c_{45}\,c_{46}^k\,(k!)^{v[\xi]}\,\|f\|_\infty^k\,\l^{1-{2\over d+1}}\int_{\SSd}\cH_{\SSd}^{d-1}(\dint u)\int_0^{\infty}\dint h\,\exp\big(-c_{47} h\big)\,.
\end{align*}
Using now fact that $1-{2\over d+1} = \frac{d-1}{d+1}$, $\int_{\SSd}\cH_{\SSd}^{d-1}(\dint u)=d\kappa_d$ and that $\int_0^\infty\dint h\,\exp(-c_{47}h)=c_{47}^{-1}$, we conclude that
\begin{align*}
\Bigg|\int_{\Delta}\dint c_\l^k\,f^{\s k}\Bigg|&\leq c_{48}\,c_{46}^k\,(k!)^{v[\xi]}\,\|f\|_\infty^k\,\l^{{d-1\over d+1}}\,.
\end{align*}
This completes the proof.
\end{proof}

In a next step we derive a first upper bound of the off-diagonal term in \eqref{eq:CumulantSumRep}. For this, we need the following auxiliary result.

\begin{lemma}\label{lem:Kombinatorik}
Let $k\geq 1$ be an integer and suppose that $k=k_1+\ldots+k_\ell$ with at most $k$ integers $k_1,\ldots,k_\ell\geq 1$. Then $k_1\cdots k_\ell\leq 4\cdot 3^k$.
\end{lemma}
\begin{proof}
Let $k=k_1+\ldots+k_\ell$ and put $K:=k_1\cdots k_\ell$. We want to maximize $K$ over all choices of $\ell$ and $k_1,\ldots,k_\ell$. For this purpose we can assume that $k_i<5$ for all $i\in\{1,\ldots,\ell\}$. Namely, if there is a factor $k_i\geq 5$, we can replace it by the two factors $k_i-2$ and $2$, which becase of $k_i<2(k_i-2)$ increases the product. We can also assume that $k_i\geq 2$, since a factor $1$ cannot contribute to the product. Moreover, to maximize $K$ we have that $k_i\leq 3$, since a factor $4$ can always be split into $2\cdot 2$ without changing the product. Thus, we have that each $k_i$ satisfies $k_i\in\{2,3\}$. Finally, we note that $3+3$ yields a bigger product than $2+2+2$. This shows that the maximal product $K$ is realized in the following way. We take $k_1=\ldots=k_\ell=3$ if $k$ is divisible by $3$, we take $k_1=2$ and $k_2=\ldots=k_\ell=3$ if $k$ leaves remainder $2$ if $k$ is divided by $3$, and we take $k_1=k_2=2$ and $k_3=\ldots=k_\ell=3$ in the remaining case. Since $\ell\leq k$ we thus have that
$$
K\leq\max\{3^{k/3},2\cdot 3^{(k-2)/3},4\cdot 3^{(k-1)/3}\}\leq 4\cdot 3^k\,,
$$
which yields the result.
\end{proof}

\begin{lemma}\label{lem4:OffDiagonal}
Let $\xi\in\Xi$ and $f\in\cB(\BBd)$. Then, for $\l\geq c_{49}$,
\begin{align*}
&\Bigg|\sum_{S,T\preceq\setk}\int_{\sigma(\{S,T\})}\dint c_\l^k\, f^{\s k}\Bigg|\leq c_{59}\,c_{51}^k\,k\,(k!)^{w[\xi]+1}\,\|f\|_\infty^k\,\l^{d-1\over d+1}\,\sum_{L_1,\ldots,L_p\preceq\setk}\int_{\SSd}\cH_{\SSd}^{d-1}(\dint u)\\
&\qquad\qquad\qquad\times \int_0^\infty\dint h_1\ldots\int_0^\infty\dint h_p\int_{(\RR^{d-1})^{p-1}}\dint{\bf v}\,\exp\big(-c_{52}\max\{\d(0,{\bf v}),h_1,\ldots,h_p\}\big)\,.
\end{align*}
\end{lemma}
\begin{proof}
We combine \eqref{eq:DecomCumulant} with the definition of the singular differential to see that
\begin{align*}
\int_{\sigma(\{S,T\})}\dint c_\l^k\,f^{\s k} &= \sum_{S',T',K_1,\ldots,K_m\preceq\setk}c(S',T',K_1,\ldots,K_m)\int_{\sigma(\{S,T\})}\tilde\dint_\l\bx\,f(\bx)\\
&\qquad\qquad\qquad\times(m_\l(\bx_{S'\cup T'})-m_\l(\bx_{S'})m_\l(\bx_{T'}))\,m_\l(\bx_{K_1})\cdots m_\l(\bx_{K_m})\,,
\end{align*}
where we also used that for each set $L\in\{S',T',S'\cup T',K_1,\ldots,K_m\}$, $m_\l(\bx_L)$ is the density of the moment measure $M_\l^{|L|}$. Now, Lemma \ref{lem2:ExpClustering} shows that
\begin{align*}
&|m_\l(\bx_{S'\cup T'})-m_\l(\bx_{S'})m_\l(\bx_{T'})| \\
&\qquad\leq c_{53}\,c_{54}^k\,k\,(|S'\cup T'|!)^{w[\xi]}\,\exp\big(-c_{55}\max\{d(v_{S'},v_{T'}),\max\{h_r:r\in S'\cup T'\}\}\big)\,,
\end{align*}
where, as usual, $\sT_\l(x_1)=(v_1,h_1),\ldots,\sT_\l(x_k)=(v_k,h_k)$.
Furthermore, conditioning on the event that for each $j\in K_i$, $\sT_\l(x_j)\in\ext(\Psi^{(\l)})$, $i\in\{1,\ldots,s\}$, we conclude similarly as in the proof of Lemma \ref{lem2:ExpClustering} that
$$
|m_\l(\bx_{K_i})| \leq c_{56}\,c_{57}^{|K_i|}\,|K_i|\,(|K_i|!)^{w[\xi]}\,\exp\big(-c_{58}\max\{h_j:j\in K_i\}\big)\,.
$$
Next, since $S'\subseteq S$ and $T'\subseteq T$, we necessarily have that $d(v_{S'},v_{T'})\geq d(v_S,v_T)$. Moreover,
$$
\max\{h_r:r\in S'\cup T'\}+\max\{h_j:j\in K_1\}+\ldots+\max\{h_j:j\in K_m\}\geq \max\{h_1,\ldots,h_k\}
$$
and thus
\begin{align*}
&\big|(m_\l(\bx_{S'\cup T'})-m_\l(\bx_{S'})m_\l(\bx_{T'}))\,m_\l(\bx_{K_1})\cdots m_\l(\bx_{K_m})\big|\\
&\qquad\leq c_{59}\,c_{60}^k\,k\,|K_1|\cdots|K_m|\,(|S'\cup T'|!)^{w[\xi]}\,(|K_1|!)^{w[\xi]}\cdots(|K_m|!)^{w[\xi]}\\
&\qquad\qquad\qquad\qquad\times\exp\big(-c_{61}\max\{d(v_S,v_T),h_1,\ldots,h_k\}\big)\,.
\end{align*}
Now, we use Lemma \ref{lem:Kombinatorik} to see that $|K_1|\cdots|K_m|\leq 4\cdot 3^k$ and we also use that
$$
(|S'\cup T'|!)^{w[\xi]}\,(|K_1|!)^{w[\xi]}\cdots(|K_m|!)^{w[\xi]}\leq (k!)^{w[\xi]}\,.
$$
This leads to
\begin{align*}
&\big|(m_\l(\bx_{S'\cup T'})-m_\l(\bx_{S'})m_\l(\bx_{T'}))\,m_\l(\bx_{K_1})\cdots m_\l(\bx_{K_m})\big|\\
&\qquad\qquad\qquad\leq c_{62}\,c_{63}^k\,k\,(k!)^{w[\xi]}\,\exp\big(-c_{64}\max\{d(v_S,v_T),h_1,\ldots,h_k\}\big)\,.
\end{align*}
Together with \eqref{eq:EstimateCoeffDecom} we have that
$$
\Bigg|\int_{\sigma(\{S,T\})}\dint c_\l^k\,f^{\s k} \Bigg|\leq c_{65}\,c_{66}^k\,(k!)^{w[\xi]+1}\,\|f\|_\infty^k\,k\,\int_{(\BBd)^k}\tilde\dint_\l\bx\,\exp\big(-c_{67}\max\{\d(\bx),h_1,\ldots,h_k\}\big)\,.
$$
What is left is to bound the integral over $(\BBd)^k$ appearing in the last expression. To evaluate it, we can and will assume without loss of generality that the point $x_1$ is mapped onto $(0,h_1)\in\cR_\l$ under $\sT_\l$ (this is possible after a suitable rotation of $\eta_\l$). Using this together with the definition of the singular differential, we conclude that
\begin{align*}
&\int_{(\BBd)^k}\tilde\dint_\l\bx\,\exp\big(-c_{67}\max\{\d(\bx),h_1,\ldots,h_k\}\big)\\
&\qquad=\sum_{L_1,\ldots,L_p\preceq\setk}\l^p\int_{(\BBd)^p}\dint(x_1,\ldots,x_p)\,\exp\big(-c_{67}\max\{d(0,v_2,\ldots,v_p),h_1,\ldots,h_p\}\big)\,.
\end{align*}
We now introduce spherical coordinates for $x_1$ and use the definition of the scaling transformation $\sT_\l$ for $x_2,\ldots,x_p$. For the differential elements $\dint x_1,\ldots,\dint x_p$ this means that
$$
\dint x_1=(1-\l^{-{2\over d+1}}h)^{d-1}\l^{-{2\over d+1}}\,\dint h_1\cH_{\SSd}^{d-1}(\dint u)
$$
and that
$$
\dint x_i=(\l^{-{1\over d+1}})^{d-1}\l^{-{2\over d+1}}\,\dint h_i\dint v_i\,,\qquad i\in\{2,\ldots,p\}\,.
$$
Together with the observation that $p-{2p\over d+1}-{1\over d+1}(d-1)(p-1)={d-1\over d+1}$, we see that
\begin{align*}
&\l^p\int_{(\BBd)^p}\dint(x_1,\ldots,x_p)\,\exp\big(-c_{67}\max\{d(0,v_2,\ldots,v_p),h_1,\ldots,h_p\}\big)\\
& = \l^p\int_{\SSd}\cH_{\SSd}^{d-1}(\dint u)\int_0^{\l^{2\over d+1}}\dint h_1\ldots\int_0^{\l^{2\over d+1}}\dint h_p\int_{\sT_\l(\SSd)}\dint v_2\ldots\int_{\sT_\l(\SSd)}\dint v_p\\
&\qquad\qquad \exp\big(-c_{67}\max\{d(0,v_2,\ldots,v_p),h_1,\ldots,h_p\}\big)\,(1-\l^{-{2\over d+1}}h)^{d-1}\l^{-{1\over d+1}(d-1)(p-1)}\l^{-{2p\over d+1}}\\
&\leq \l^{d-1\over d+1}\int_{\SSd}\cH_{\SSd}^{d-1}(\dint u)\int_0^{\infty}\dint h_1\ldots\int_0^{\infty}\dint h_p\int_{(\RR^{d-1})^{p-1}}\dint{\bf v}\\
&\qquad\qquad\exp\big(-c_{67}\max\{d(0,{\bf v}),h_1,\ldots,h_p\}\big)\,.
\end{align*}
This yields the result.
\end{proof}

Fix from now on and until Lemma \ref{lem7:IntEstimate} a partition $\{L_1,\ldots,L_p\}$ of $\setk$. Our next goal is to bound the integral
$$
\int_{(\RR^{d-1})^{p-1}}\dint{\bf v}\,\exp\big(-c\,\max\{\d(0,{\bf v}),h_1,\ldots,h_p\}\big)
$$
that has shown up in Lemma \ref{lem4:OffDiagonal}, where $c\in(0,\infty)$ is a constant only depending on $d$ and on $\xi$. We put $a:=\max\{h_1,\ldots,h_p\}$ and write
\begin{align}
\nonumber & \int_{(\RR^{d-1})^{p-1}}\dint{\bf v}\,\exp\big(-c\,\max\{\d(0,{\bf v}),h_1,\ldots,h_p\}\big)\\
&\qquad=\int_{\{\d(0,{\bf v})\leq a\}}\dint{\bf v}\,\exp(-c\,a) + \int_{\{\d(0,{\bf v})>a\}}\dint{\bf v}\,\exp(-c\,\d(0,{\bf v})) =: T_1 + T_2\,.\label{eq:defT1T2}
\end{align}

\begin{lemma}\label{lem5:T1}
For $T_1$ defined at \eqref{eq:defT1T2} we have that
$$
T_1 \leq \exp(-c\,a)\,p^{p-2}\,\kappa_{d-1}^{p-1}\,a^{(p-1)(d-1)}\,.
$$
\end{lemma}
\begin{proof}
Suppose that $\d(0,{\bf v})\leq a$. Then there is obviously no partition $\{S,T\}$ of $\{0,v_2,\ldots,v_p\}$ such that the corresponding separation $d({\bf v}_S,{\bf v}_T)$ is bigger than $a$. This implies that there exists a tree $\cT$ on $\{1,\ldots,p\}$ such that adjacent vertices $v_i,v_j$ in $\cT$ satisfy $\|v_i-v_j\|\leq a$. We indicate this property by writing $(0,{\bf v})\Yleft (a,\cT)$ and thus have
\begin{align*}
\int_{\{\d(0,{\bf v})\leq a\}}\dint{\bf v}&=\sum_{\cT\in{\rm Trees}(\setp)}\int_{\{(0,{\bf v})\Yleft (a,\cT)\}}\dint{\bf v}\\
&= \sum_{\cT\in{\rm Trees}(\setp)}V_{(d-1)(p-1)}(\{{\bf v}\in(\RR^{d-1})^{p-1}:(0,{\bf v})\Yleft (a,\cT)\})\,,
\end{align*}
where the sums run over all trees $\cT$ on the set $\setp$. By the geometry of these trees mentioned above, we have that
$$
V_{(d-1)(p-1)}(\{{\bf v}\in(\RR^{d-1})^{p-1}:(0,{\bf v})\Yleft (a,\cT)\})\leq (a^{d-1}\kappa_{d-1})^{p-1}\,.
$$
Moreover, by Caley's theorem there are exactly $p^{p-2}$ trees on $\setp$. This yields
$$
\int_{\{\d(0,{\bf v})\leq a\}}\dint{\bf v}\leq p^{p-2}\,(a^{d-1}\kappa_{d-1})^{p-1}\,.
$$
Multiplication with $\exp(-c\,a)$ completes the proof.
\end{proof}

\begin{lemma}\label{lem6:T2}
For $T_2$ defined at \eqref{eq:defT1T2} we have that
$$
T_2 \leq p^{p-2}\,\kappa_{d-1}^{p-1}\sum_{i=0}^{(d-1)(p-1)}{((d-1)(p-1))_{(i)}\over c^i}\,\exp(-c\,a)\,a^{(d-1)(p-1)-i}\,.
$$
\end{lemma}
\begin{proof}
We rewrite $T_2$ as
\begin{align*}
T_2  &= c\int_{\{\d(0,{\bf v})>a\}}\dint{\bf v}\int_{\d(0,{\bf v})}^\infty\dint t\, \exp(-c\,t) = c\int_a^\infty\dint t\int_{\{a<\d(0,{\bf v})<t\}}\dint{\bf v}\, \exp(-c\,t)\\
&\leq c\int_a^\infty\dint t\int_{\{\d(0,{\bf v})<t\}}\dint{\bf v}\, \exp(-ct)\,.
\end{align*}
Arguing in the same manner as in the proof of Lemma \ref{lem5:T1} we see that
$$
T_2\leq c\,p^{p-2}\kappa_{d-1}^{p-1}\int_a^\infty\dint t\,t^{(d-1)(p-1)}\,\exp(-c\,t)\,.
$$
Applying Lemma \ref{lem1:Integral} to the last integral yields that
\begin{align*}
T_2 &\leq c\,p^{p-2}\,\kappa_{d-1}^{p-1}\,\sum_{i=1}^{(d-1)(p-1)+1}{((d-1)(p-1))_{(i-1)}\over c^i}\,\exp(-ca)\,a^{(d-1)(p-1)-(i-1)}\\
&=c\,p^{p-2}\,\kappa_{d-1}^{p-1}\,\sum_{i=0}^{(d-1)(p-1)}{((d-1)(p-1))_{(i)}\over c^{i+1}}\,\exp(-ca)\,a^{(d-1)(p-1)-i}\\
&=p^{p-2}\,\kappa_{d-1}^{p-1}\,\sum_{i=0}^{(d-1)(p-1)}{((d-1)(p-1))_{(i)}\over c^{i}}\,\exp(-ca)\,a^{(d-1)(p-1)-i}\,.
\end{align*}
This completes the proof.
\end{proof}

\begin{remark}\label{rem:Integral}\rm 
The proof of Lemma \ref{lem6:T2} shows why in Lemma \ref{lem2:ExpClustering} the exponential estimates \eqref{eq:ExpStabilization2} and \eqref{eq:ExtremExpDecay2} are used instead of \eqref{eq:ExpStabilization} and \eqref{eq:ExtremExpDecay}, respectively. Using the latter estimates would lead to
$$
T_2\leq c\,p^{p-2}\kappa_{d-1}^{p-1}\int_a^\infty\dint t\,t^{(d-1)(p-1)\over d+1}\,\exp(-c\,t)\,.
$$
However, we do not know about a closed form expression for the last integral that could be used in the further steps of our proof.
\end{remark}

Combining Lemma \ref{lem5:T1} and Lemma \ref{lem6:T2} we conclude that
\begin{align}\label{eq:T1+T2}
\nonumber &\int_{(\RR^{d-1})^{p-1}}\dint{\bf v}\,\ext\big(-c\,\max\{\d(0,{\bf v}),h_1,\ldots,h_p\}\big)\\
\nonumber &\leq p^{p-2}\,\kappa_{d-1}^{p-1}\Big(\exp(-ca)\,a^{(d-1)(p-1)}+\sum_{i=0}^{(d-1)(p-1)}{((d-1)(p-1))_{(i)}\over c^{i}}\,\exp(-ca)\,a^{(d-1)(p-1)-i}\Big)\\
&\leq 2p^{p-2}\,\kappa_{d-1}^{p-1}\sum_{i=0}^{(d-1)(p-1)}{((d-1)(p-1))_{(i)}\over c^i}\,\exp(-c\,a)\,a^{(d-1)(p-1)-i}
\end{align}
for the innermost integral appearing in Lemma \ref{lem4:OffDiagonal}. We now carry out the integration with respect to the height coordinates $h_1,\ldots,h_p$.

\begin{lemma}\label{lem7:IntEstimate}
We have that
\begin{align*}
&\int_0^\infty\dint h_1\ldots\int_0^\infty\dint h_p\int_{(\RR^{d-1})^{p-1}}\dint{\bf v}\,\exp\big(-c\,\max\{\d(0,{\bf v}),h_1,\ldots,h_p\}\big)\\
&\qquad\leq c_{68}\,c_{69}^p\,p^{2(p-1)}\,(dp)!\,.
\end{align*}
\end{lemma}
\begin{proof}
From \eqref{eq:T1+T2} we conclude that
\begin{equation}\label{eq:Lem7Step}
\begin{split}
&\int_0^\infty\dint h_1\ldots\int_0^\infty\dint h_p\int_{(\RR^{d-1})^{p-1}}\dint{\bf v}\,\exp\big(-c\,\max\{\d(0,{\bf v}),h_1,\ldots,h_p\}\big)\\
&\leq 2p^{p-2}\kappa_{d-1}^{p-1}\sum_{i_1=0}^{(d-1)(p-1)}{((d-1)(p-1))_{(i_1)}\over c^{i_1}}\int_{0}^\infty\dint h_1\ldots\int_0^\infty\dint h_p\,\exp(-c\,a)\,a^{(d-1)(p-1)-{i_1}}\,,
\end{split}
\end{equation}
where, recall, $a=\max\{h_1,\ldots,h_p\}$. Define $a_1:=\max\{h_2,\ldots,h_p\}$ and note that
\begin{align*}
&\int_0^\infty\dint h_1\,\exp(-c\,a)\,a^{(d-1)(p-1)-{i_1}}\\
&\qquad = \int_0^{a_1}\dint h_1\,\exp(-c\,a_1)\,a_1^{(d-1)(p-1)-{i_1}}+\int_{a_1}^\infty\dint h_1\,\exp(-c\,h_1)\,h_1^{(d-1)(p-1)-{i_1}}\,.
\end{align*}
The first integral is just $\exp(-c\,a_1)\,a_1^{(d-1)(p-1)-{i_1}+1}$ and the second one can be evaluated by means of Lemma \ref{lem1:Integral}. This yields
\begin{align*}
&\int_0^\infty\dint h_1\,\exp(-c\,a)\,a^{(d-1)(p-1)-{i_1}}\\
&\qquad\leq \exp(-c\,a_1)\,a_1^{(d-1)(p-1)-{i_1}+1}\\
&\qquad\qquad\qquad+\sum_{i_2=1}^{(d-1)(p-1) - i_1 +1} \frac{((d-1)(p-1) - i_1)_{(i_2 -1)}}{c^{i_2}} \exp(-c\, a_1)\, a_1^{(d-1)(p-1) - i_1 - (i_2-1)}\\
&\qquad =\sum_{i_2=0}^{(d-1)(p-1) - i_1 +1} \frac{((d-1)(p-1) - i_1)_{(i_2 -1)}}{c^{i_2}} \exp(-c\, a_1)\, a_1^{(d-1)(p-1) - (i_1 + i_2) +1}\,.
\end{align*}
In a next step we integrate the individual summands appearing in the last line with respect to $h_2$. Putting $a_2:=\max\{h_2,\ldots,h_p\}$ we conclude, similarly as above, that
\begin{align*}
&\int_{0}^{\infty}\dint h_2\,\exp(-c\, a_1) a_1^{(d-1)(p-1) - (i_1 + i_2) + 1} \\
&=\int_{0}^{a_2}\dint h_2\,\exp(-c \, a_2)\, a_2^{(d-1)(p-1) - (i_1 + i_2) + 1} +  \int_{a_2}^{\infty}\dint h_2\,\exp(-c \, h_2) \, h_2^{(d-1)(p-1) - (i_1 + i_2) + 1} dh_2\\
&= \sum_{i_3=0}^{(d-1)(p-1) - (i_1 + i_2) +2} \frac{((d-1)(p-1) - (i_1 + i_2) + 1)_{(i_3 -1)}}{c^{i_3}}\, \exp(-c \, a_2) a_2^{(d-1)(p-1) - (i_1 + i_2 + i_3) +2}\,.
\end{align*}
This procedure can now be repeated $p-3$ further times and in the very last step the resulting integral
$$
\mathcal{J}:=\int_0^\infty\dint h_p\,\exp(-c\,h_p)\,h_p^{(d-1)(p-1)-(i_1+\ldots+i_p)+(p-1)}
$$
can be evaluated explicitly:
\begin{align*}
\mathcal{J} &= {((d-1)(p-1)-(i_1+\ldots+i_p)+(p-1))!\over c^{(d-1)(p-1)-(i_1+\ldots+i_p)+(p-1)}}\int_0^\infty\dint h_p\,\exp(-c\,h_p)\\
&={((d-1)(p-1)-(i_1+\ldots+i_p)+(p-1))!\over c^{(d-1)(p-1)-(i_1+\ldots+i_p)+p}}\,.
\end{align*}
This leads in view of \eqref{eq:Lem7Step} to
\begin{align*}
&\int_0^\infty\dint h_1\ldots\int_0^\infty\dint h_p\int_{(\RR^{d-1})^{p-1}}\dint{\bf v}\,\exp\big(-c\,\max\{\d(0,{\bf v}),h_1,\ldots,h_p\}\big)\\
&\leq 2p^{p-2}\,\kappa_{d-1}^{p-1}\sum_{i_1=0}^{(d-1)(p-1)}{((d-1)(p-1))_{(i_1)}\over c^{i_1}}\sum_{i_2=0}^{(d-1)(p-1)-i_1+1}{((d-1)(p-1)-i_1)_{(i_2-1)}\over c^{i_2}}\\
&\quad\times\ldots\times\sum_{i_p=0}^{(d-1)(p-1)-(i_1+\ldots+i_{p-1})+(p-1)}{((d-1)(p-1)-(i_1+\ldots+i_{p-1})+(p-2))_{(i_p-1)}\over c^{i_p}}\\
&\quad\times{((d-1)(p-1)-(i_1+\ldots+i_p)+(p-1))!\over c^{(d-1)(p-1)-(i_1+\ldots+i_p)+p}}\\
& \leq c_{70}\,c_{71}^p\,p^{p-2} \sum_{i_1=0}^{(d-1)(p-1)} ((d-1)(p-1))_{(i_1)} \sum_{i_2=0}^{(d-1)(p-1) - i_1 +1} ((d-1)(p-1) - i_1)_{(i_2 -1)}\\
&\quad \times \cdots \times \sum_{i_p=0}^{(d-1)(p-1) - (i_1 + \cdots + i_{p-1}) +(p-1)} ((d-1)(p-1) - (i_1 + \cdots + i_{p-1}) + (p-2))_{(i_p -1)}\\
&\quad \times \left((d-1)(p-1) - (i_1 + \cdots + i_p) +(p-1)\right)!
\end{align*}
with $c_{70}=2$ and $1/c^{(d-1)(p-1)+p}\leq 1/c^p=:c_{71}^p$.
Now, we expand and notice that each of the resulting terms is a product of $p+1$ falling factorials. From the structure of these $p+1$ factors it follows that each product that shows up this way is bounded from above by $((d-1)(p-1) + (p-1))! = (d(p-1))!$ (this corresponds precisely to the last factorial with $i_1=\cdots =i_p=0$). Thus,
\begin{align*}
&\int_0^\infty\dint h_1\ldots\int_0^\infty\dint h_p\int_{(\RR^{d-1})^{p-1}}\dint{\bf v}\,\exp\big(-c\,\max\{\d(0,{\bf v}),h_1,\ldots,h_p\}\big)\\
&\leq c_{71}\,c_{72}^p \, p^{p-2} \,(d(p-1))! \sum_{i_1=0}^{(d-1)(p-1)} \sum_{i_2=0}^{(d-1)(p-1) - i_1 +1} \cdots \sum_{i_p=0}^{(d-1)(p-1) - (i_1 + \cdots + i_{p-1}) +(p-1)} 1\\
&\leq c_{71}\,c_{72}^p \,p^{p-2} \,(d(p-1))! \sum_{i_1=0}^{(d-1)(p-1)} \sum_{i_2=0}^{(d-1)(p-1) +1} \cdots \sum_{i_p=0}^{(d-1)(p-1) +(p-1)} 1\\
&\leq c_{71}\,c_{72}^p \, p^{p-2} \, (d(p-1))! \left[(d-1)(p-1)+ 1\right] \left[(d-1)(p-1) +2\right]\cdots \left[(d-1)(p-1) +p\right]\\
&\leq c_{71}\,c_{72}^p \, p^{p-2} \, (d(p-1))!\, (dp)^{p}\\
&\leq c_{71}\,c_{73}^p\, p^{2(p-1)}\,(dp)!
\end{align*}
and we have completed the proof.
\end{proof}

Now, we use Lemma \ref{lem7:IntEstimate} to derive our final upper bound for the off-diagonal term in \eqref{eq:CumulantSumRep}.

\begin{lemma}\label{lem8:OffDiagEst}
For $\xi\in\Xi$ and $f\in\cB(\BBd)$ we have that
$$
\Bigg|\sum_{S,T\preceq\setk}\int_{\sigma(\{S,T\})}\dint c_\l^k\, f^{\s k}\Bigg| \leq c_{74}\,c_{75}^k\,\|f\|_\infty^k\,(k!)^{w[\xi]+2}\,k^{2k}\,(dk)!\,\l^{d-1\over d+1}
$$
for all $\l\geq c_{76}$.
\end{lemma}
\begin{proof}
We combine Lemma \ref{lem4:OffDiagonal} with Lemma \ref{lem7:IntEstimate} to conclude that
\begin{align*}
&\Bigg|\sum_{S,T\preceq\setk}\int_{\sigma(\{S,T\})}\dint c_\l^k\, f^{\s k}\Bigg|\\
&\qquad\leq c_{77}\,c_{78}^k\,(k!)^{w[\xi]+1}\,\|f\|_\infty^k\,k\,\l^{d-1\over d+1}\sum_{L_1,\ldots,L_p\preceq\setk}\int_{\SSd}\cH_{\SSd}^{d-1}(\dint u)\,c_{79}\,c_{80}^p\,p^{2(p-1)}\,(dp)!\,.
\end{align*}
The number of terms in the sum is bounded by $k!$ and the integral over $\SSd$ equals $d\k_d$. Thus, using that $p\leq k$,
\begin{align*}
\Bigg|\sum_{S,T\preceq\setk}\int_{\sigma(\{S,T\})}\dint c_\l^k\, f^{\s k}\Bigg|&\leq c_{81}\,c_{82}^k\,\|f\|_\infty^k\,(k!)^{w[\xi]+2}\,k\,p^{2(p-1)}\,(dp)!\,\l^{d-1\over d+1}\\
&\leq c_{81}\,c_{83}^k\,\|f\|_\infty^k\,(k!)^{w[\xi]+2}\,k^{1+2(k-1)}\,(dk)!\,\l^{d-1\over d+1}\\
&\leq c_{81}\,c_{83}^k\,\|f\|_\infty^k\,(k!)^{w[\xi]+2}\,k^{2k}\,(dk)!\,\l^{d-1\over d+1}\,.
\end{align*}
This completes the argument.
\end{proof}

We are now prepared to show our crucial cumulant bound.

\begin{proof}[Proof of Proposition \ref{prop:CumulantEstimate}.]
From Lemma \ref{lem3:Diagonal} and Lemma \ref{lem8:OffDiagEst} we conclude that 
\begin{align*}
|\lan f^{\s k},c_\l^k\ran| & \leq\Bigg|\int_{\Delta} \dint c_\l^k\, f^{\s k}\Bigg|+\Bigg|\sum_{S,T\preceq\setk}\int_{\sigma(\{S,T\})}\dint c_\l^k\, f^{\s k}\Bigg|\\
&\leq c_{84}\,c_{85}^k\,\|f\|_\infty^k\,(k!)^{\max\{w[\xi]+2,v[\xi]\}}\,k^{2k}\,(dk)!\,\l^{d-1\over d+1}
\end{align*}
for all sufficiently large $\l$. Now, the elementary inequality $\ell^\ell\leq \ell!\,e^{3\ell}$, valid for $\ell\in\{3,4,\ldots\}$, yields that $k^{2k}\leq (k!)^2\,e^{6k}$. Furthermore, we have from Stirling's formula that
\begin{align*}
(dk)! \leq e\,(dk)^{dk+{1\over 2}}\,e^{-dk} &= e\,\sqrt{dk}\,e^{-dk}\,d^{dk}\,(k^k)^d \leq e\,\sqrt{dk}\,e^{-dk}\,d^{dk}\,e^{3dk}\,(k!)^d\\
&\leq e^{4dk}\,d^{dk}\,(k!)^d = (d^d\,e^{4d})^k\,(k!)^d\,.
\end{align*}
Altogether this implies that
\begin{align*}
|\lan f^{\s k},c_\l^k\ran| &\leq c_{86}\,c_{87}^k\,\|f\|_\infty^k\,(k!)^{d+\max\{v[\xi],w[\xi]+2\}+2}\,\l^{d-1\over d+1}\\
&\leq c_{86}\,c_{87}^k\,\|f\|_\infty^k\,(k!)^{d+w[\xi]+4}\,\l^{d-1\over d+1}
\end{align*}
for all $\l\geq c_{88}$ and completes the proof.
\end{proof}

\begin{remark}\rm 
Our proof shows that, in principle, our results in Section \ref{sec:MainResults} apply to any functional $\xi$ acting on pairs $(x,\eta_\l)$ with $x\in\eta_\l$ that satisfy the exponential localization property \eqref{eq:ExpStabilization}, the exponential decay property \eqref{eq:ExtremExpDecay}, which have moments of all orders as in Lemma \ref{lem:MomentsAllOrders} and for which the variance lower bound condition \eqref{eq:VarLowerBound} is satisfied. We decided to restrict to the key geometric functionals $\xi\in\Xi$ as introduced in Section \ref{sec:KeyFunctionals}, because for these examples the above assumptions can be verified and since they cover the most interesting and most fundamental quantities associated with the random polytopes $\Pi_\l$.
\end{remark}

\subsection{Proof of the theorems}

We are now prepared to establish our main results presented in Section \ref{sec:MainResults}. For this, we need the following lemma, which is included to make the paper self-contained. By slight abuse of notation, we denote by $c^k(X)$ the $k$th cumulant of a (real-valued) random variable $X$. It is well defined if $\EE|X|^k<\infty$ and is given by
$$
c^k(X) := (-\mathfrak{i})^{k}\,{\dint^k\over\dint t^k}\log\EE e^{\mathfrak{i}tX}\Big|_{t=0}\,,
$$
where $\mathfrak{i}$ stands for the imaginary unit.

\begin{lemma}\label{lem:Litauer}
Let $(X_\l)_{\l>0}$ be a family of random variables with $\EE X_\l=0$ and $\var X_\l=1$ for all $\l>0$, and suppose that, for all $\l\geq\l_0$,
$$
|c^k(X_\l)| \leq (k!)^{1+\gamma}\,\Delta_\l^{-(k-2)}
$$
for some $\l_0>0$, $\gamma\in[0,\infty)$, $\Delta_\l\in(0,\infty)$ and all $k\in\{3,4,\ldots\}$.
\begin{itemize}
\item[(i)] There exists $b_4\in(0,\infty)$ only depending on $\gamma$ such that
$$
\PP(|X_\l|\geq y) \leq 2\exp\Big(-{1\over 4}\min\Big\{{y^2\over 2^{1+\gamma}},(\Delta_\l y)^{1\over 1+\gamma}\Big\}\Big)
$$
for all $y\geq 0$ and $\l\geq b_4$.

\item[(ii)] There are constants $b_5,b_6,b_7\in(0,\infty)$ only depending on $\gamma$ such that for $\Delta_\l\geq b_5$ and $0\leq y\leq b_6\Delta_\l^{1/(1+2\gamma)}$,
\begin{align*}
\Bigg|\log{\PP(X_\l\geq y)\over 1-\Phi(y)}\Bigg| &\leq b_7\,(1+y^3)\,\Delta_\l^{-1/(1+2\gamma)}\quad\text{and}\\
\Bigg|\log{\PP(X_\l\leq -y)\over \Phi(-y)}\Bigg| &\leq b_7\,(1+y^3)\, \Delta_\l^{-1/(1+2\gamma)}\,,
\end{align*}
where $\Phi(\,\cdot\,)$ is the distribution function of a standard Gaussian random variable.

\item[(iii)] If $(a_\l)_{\l>0}$ is such that
$$
\lim_{\l\to\infty} a_\l = \infty \qquad\text{and}\qquad \lim_{\l\to\infty}a_\l\,\Delta_\l^{-1/(1+2\gamma)}\,,
$$
then $(a_\l^{-1}X_\l)_{\l>0}$ satisfies a moderate deviation principle on $\RR$ with speed $a_\l^2$ and rate function $I(y)={y^2\over 2}$.
\end{itemize}
\end{lemma}
\begin{proof}
Part (i) is a reformulation of Lemma 2.4 in \cite{SaulisBuch} in a form taken from Lemma 3.9 in \cite{EichelsbacherSchreiberRaic} with $H=2^{1+\gamma}$ there. The statement in (ii) corresponds to Lemma 2.3 in \cite{SaulisBuch} in a form that we took from Corollary 3.2 in \cite{EichelsbacherSchreiberRaic}. Finally, the MDP for the family $(X_\l)_{\l>0}$ is Theorem 1.1 in \cite{DoeringEichelsbacher}.
\end{proof}

We now combine the previous lemma with the cumulant bound established in Proposition \ref{prop:CumulantEstimate} to give a proof of our main results for Poisson polytopes.

\begin{proof}[Proof of Theorem \ref{thm:LDIGeneral}, Theorem \ref{thm:MDProbasGeneral} and Theorem \ref{thm:MDPScalarGeneral}.]
We let $\xi\in\Xi$ be a key geometric functional and $f\in\cC(\BBd)$ with $\lan f^2,\cH_{\SSd}^{d-1}\ran\neq 0$. Recalling from \eqref{eq:VarLowerBound} that in this case $\sigma_\l^\xi[f]\geq c_{89}\lan f^2,\cH_{\SSd}^{d-1}\ran^{1\over 2}\l^{d-1\over 2(d+1)}$, we see in view of Proposition \ref{prop:CumulantEstimate} that, for $k\in\{3,4,\ldots\}$,
\begin{align}\label{eq:FinalCumEst}
\nonumber {|\lan f^{\s k},c_\l^k\ran|\over(\sigma_\l^\xi[f])^k} &\leq c_{2}\,c_3^k\,\|f\|_\infty^k\,\l^{d-1\over d+1}\,(k!)^{d+w[\xi]+4}\,\Big(c_{90}\,\lan f^2,\cH_{\SSd}^{d-1}\ran^{1\over 2}\,\l^{d-1\over 2(d+1)}\Big)^{-k}\\
\nonumber &\leq (k!)^{d+w[\xi]+4}\,c_2\,\l^{d-1\over d+1}\Bigg({c_3\,\|f\|_\infty\over c_{90}\,\lan f^2,\cH_{\SSd}^{d-1}\ran^{1\over 2 }\,\l^{d-1\over 2(d+1)}}\Bigg)^{k} \\
&= (k!)^{d+w[\xi]+4}\,\Big(c_{93}\,\l^{d-1\over 2(d+1)}\Big)^{-(k-2)}\,,
\end{align}
where $c_{93}:=c_{91}(c_3\|f\|_\infty\max\{1,c_{92}\})^{-1}$ with $c_{91}:=c_{90}\lan f^2,\cH_{\SSd}^{d-1}\ran^{1\over 2 }$ and $c_{92}:=c_2c_3^2c_{91}^{-2}\|f\|_\infty^2$ is a constant depending only on $d$, $\xi$ and $f$.
Now, put
\begin{equation}\label{eq:GammaDeltaCon}
\gamma := d+w[\xi]+3\qquad\text{and}\qquad \Delta_\l:=c_{93}\,\l^{d-1\over 2(d+1)}
\end{equation}
and apply Lemma \ref{lem:Litauer} to the random variables $X_\l:=(\sigma_\l^\xi[f])^{-1}\,\lan f,\l^{e[\xi]}\bar\mu_\l^\xi\ran$. The results then follow and the proof is complete.
\end{proof}

\begin{proof}[Proof of Theorem \ref{thm:MDPMeasureGeneral}]
Theorem \ref{thm:MDPMeasureGeneral} follows by the same argument as Theorem 1.5 in \cite{EichelsbacherSchreiberRaic} follows from Theorem 1.4 ibidem with the class $\cB(\BBd)$ replaced by $\cC(\SSd)$. For this reason, details are omitted.
\end{proof}

\appendix
\section{Appendix}

Recall that for a point $x\in\BBd\setminus\{0\}$, $L(x)$ is the line through $x$ and $0$. Moreover, $G(L(x),j)$, $j\in\{1,\ldots,d-1\}$, is the space of $j$-dimensional linear subspaces of $\RR^d$ containing $L(x)$, which is supplied with the Haar probability measure $\nu_j^{L(x)}$. Further recall the definition of $\theta_j(\,\cdot\,)$ from \eqref{eq:defThetak}. 

\begin{lemma}\label{lem:RepVK}
Let $K\subset\BBd$ be a convex body. Then, for $j\in\{1,\ldots,d-1\}$, one has that $$V_j(\BBd)-V_j(K)={{d-1\choose j-1}\over\k_{d-j}}\int_{\BBd\setminus K}\dint x\, \theta_j(x)\,\|x\|^{-(d-j)}\,.$$
\end{lemma}
\begin{proof}
The mean projection formula from integral geometry \cite[Theorem 6.2.2]{SW} asserts that
\begin{align*}
V_j(\BBd)-V_j(K) &= {d\choose j}{\k_d\over\k_j\k_{d-j}}\int_{G(d,j)}\nu_j(\dint L)\,\big(V_j(\BBd|L)-V_j(K|L)\big)\\
&={d\choose j}{\k_d\over\k_j\k_{d-j}}\int_{G(d,j)}\nu_j(\dint L)\int_{\BB_L}\dint x\,\big(1-{\bf 1}(x\notin K|L)\big)\,,
\end{align*}
where $\BB_L=\BBd|L$ is the unit ball in $L$. To the inner integral over $\BB_L$ we apply the Blaschke-Petkanschin formula \cite[Theorem 7.2.1]{SW}, which gives
\begin{align*}
&{d\choose j}{\k_d\over\k_j\k_{d-j}}\int_{G(d,j)}\nu_j(\dint L)\int_{\BB_L}\dint x\,\big(1-{\bf 1}(x\notin K|L)\big)\\
&={j\k_j\over 2}{d\choose j}{\k_d\over\k_j\k_{d-j}}\int_{G(d,j)}\nu_j(\dint L)\int_{G(L,1)}\nu_1^L(\dint M)\int_{\BB_M}\dint x\,\big(1-{\bf 1}(x\notin K|L)\big)\,\|x\|^{j-1}\,.
\end{align*}
Applying \cite[Theorem 7.1.2]{SW} and Fubini's theorem, the last expression is transformed into 
\begin{align*}
&{j\k_j\over 2}{d\choose j}{\k_d\over\k_j\k_{d-j}}\int_{G(d,1)}\nu_1(\dint M)\int_{G(M,j)}\nu_j^M(\dint L)\int_{\BB_M}\dint x\,\big(1-{\bf 1}(x\notin K|L)\big)\,\|x\|^{j-1}\\
&={j\k_j\over 2}{d\choose j}{\k_d\over\k_j\k_{d-j}}\int_{G(d,1)}\nu_1(\dint M)\int_{\BB_M}\dint x\int_{G(M,j)}\nu_j^M(\dint L)\,\big(1-{\bf 1}(x\notin K|L)\big)\,\|x\|^{j-1}\\
&={j\k_j\over 2}{d\choose j}{\k_d\over\k_j\k_{d-j}}\int_{G(d,1)}\nu_1(\dint M)\int_{\BB_M}\dint x\int_{G(L(x),j)}\nu_j^{L(x)}(\dint L)\,\big(1-{\bf 1}(x\notin K|L)\big)\,\|x\|^{j-1}\,,
\end{align*}
where we have also used that $M=L(x)$. To this expression we apply the Blaschke-Petkanschin formula \cite[Theorem 7.2.1]{SW} now backwards to see that
\begin{align*}
&{j\k_j\over 2}{d\choose j}{\k_d\over\k_j\k_{d-j}}\int_{G(d,1)}\nu_1(\dint M)\int_{\BB_M}\dint x\int_{G(L(x),j)}\nu_j^{L(x)}(\dint L)\,\big(1-{\bf 1}(x\notin K|L)\big)\,\|x\|^{j-1}\\
&={j\k_j\over d\k_d}{d\choose j}{\k_d\over\k_j\k_{d-j}}\int_{\BBd}\dint x\int_{G(L(x),j)}\nu_j^{L(x)}(\dint L)\,\big(1-{\bf 1}(x\notin K|L)\big)\,{\|x\|^{j-1}\over \|x\|^{d-1}}\\
&={{d-1\choose j-1}\over\k_{d-j}}\int_{\BBd\setminus K}\dint x\int_{G(L(x),j)}\nu_j^{L(x)}(\dint L)\,\big(1-{\bf 1}(x\notin K|L)\big)\,\|x\|^{-(d-j)}\,.
\end{align*}
Taking into account the definition of $\theta_j(x)$ completes the proof.
\end{proof}

\subsection*{Acknowledgements}
We would like to thank Matthias Reitzner (University of Osnabr\"uck) for helpful discussions and for useful hints and comments to an earlier version of this paper.


\end{document}